%% file: a_new_proof_to_period_of_GL_2_.tex
\documentclass{amsart}
\usepackage{amssymb
,amsthm
,amsmath
,amscd
,mathtools,
}
\usepackage{enumerate}
\usepackage{hyperref}
\usepackage[all]{xy}
\usepackage[top=30truemm,bottom=30truemm,left=25truemm,right=25truemm]{geometry}

\newcommand{\Sp}{\mathrm{Sp}}

\newcommand{\GL}{\mathrm{GL}}

\newcommand{\Oo}{\mathrm{O}}
\newcommand{\SO}{\mathrm{SO}}
\newcommand{\GSp}{\mathrm{GSp}}
\newcommand{\GSO}{\mathrm{GSO}}
\newcommand{\PGL}{\mathrm{PGL}}
\newcommand{\GO}{\mathrm{GO}}

\newcommand{\Hom}{\mathrm{Hom}}
\newcommand{\Gal}{\mathrm{Gal}}
\newcommand{\Ext}{\mathrm{Ext}}
\newcommand{\BC}{\mathrm{BC}}

\newtheorem{thm}{Theorem}[section]
\newtheorem{lem}[thm]{Lemma}
\newtheorem{prop}[thm]{Proposition}
\newtheorem{coro}[thm]{Corollary}
\newtheorem{rem}[thm]{Remark}

\makeatletter
\def\iddots{\mathinner{\mkern1mu\raise\p@
	\hbox{.}\mkern2mu\raise4\p@\hbox{.}\mkern2mu
	\raise7\p@\vbox{\kern7\p@\hbox{.}}\mkern1mu}}
\def\adots{\mathinner{\mkern2mu\raise\p@\hbox{.}
 \mkern2mu\raise4\p@\hbox{.}\mkern1mu
 \raise7\p@\vbox{\kern7\p@\hbox{.}}\mkern1mu}}
\makeatother

\allowdisplaybreaks
\title{A New Proof to the Period Problems of $\GL_2$}
\author{LU Hengfei}
\address{Department of Mathematics, National University of Singapore, 10 Lower Kent Ridge Road, Singapore 119076}
\email{luhengfei@u.nus.edu}
\begin{document}
%\maketitle

\begin{abstract}
We use the relations among the base change representations, theta lifts and Whittaker model, to give a new proof to the  period problems of $\GL_2$ over a quadratic extension $E/F$ of local fields. We classify both
local and global $D^\times(F)$-distinguished representations $\pi^D$ of $D^\times(E),$ where $D^\times$ is an inner form of $\GL_2$ defined over a nonarchimedean field or a number field $F.$ 
\end{abstract}

\maketitle
\tableofcontents
\input{introd}
\subsection*{Acknowledgments} This is part of the author's Ph.D thesis \cite{lu2016}.
The author is grateful to Professor Gan  Wee Teck for his guidance and numerous discussions. The author also thanks the anonymous referees for the careful reading and helpful comments for the earlier version.

\input{preliminaries}

\input{local}
\input{squareintegral3}
\input{steinberg2}
\input{super}

\input{v-}

\input{prooflocal}
\input{global}
\input{global2}
\bibliographystyle{amsalpha}
\bibliography{draft}
\end{document}

%% file: introd.tex
\section{Introduction}
Period problems,
which are closely related to Harmonic Analysis, have been extensively studied for classical groups. The most general situations have been studied  in  \cite{sakellaridis2012periods}, and we will focus on the period problems of $\GL_2$  in this paper.
\par
Assume that $F$ is a nonarchimedean local field of characteristic $0.$
Let $G$ be a connected reductive group defined over $F$ and $H$
be a subgroup of $G.$ Given a smooth irreducible representation $\pi$ of $G(F),$ one may consider $\dim \Hom_{H(F)}(\pi,\mathbb{C}).$ If it is nonzero, then we say that $\pi$ is  $H(F)$-distinguished, or has a nonzero $H(F)$-period. 
One may also consider the Ext version $\Ext^1_{H(F)}(\pi,\mathbb{C})$ in the category of smooth representations of $H(F)$ with trivial central characters.
% i.e. $Ext^1_{H(F)}(\pi,\mathbb{C} ).$ 
\par
Let $W_F$ be the Weil group of $F$ and  $WD_F$ be the Weil-Deligne group. 
Assume that $\tau$ is an irreducible smooth representation of $\GL_2(F),$
 with Langlands parameter %for $GL_2,$
$\phi_\tau:WD_F\longrightarrow \GL_2(\mathbb{C}).$
Assume that $E$ is a quadratic field extension of $F.$ Then $\phi_\tau|_{WD_E}$
corresponds to an irreducible admissible representation of $\GL_2(E),$ which is denoted by $\BC(\tau).$ Then
%Then we can talk about the period problem.
%Assume $\pi$ is an admissible representation of $GL_2(E),$
%$\chi$ is a character of $GL_2(F),$ regarded as a subgroup of $GL_2(E).$
%$\pi$ has a nonzero $(GL_2(F),\chi)$-period if there is a nonzero intertwining operater from $\pi|_{GL_2(F)}$
%to $\chi.$
%If %$Hom_{GL_2(F)}(\pi,\mathbb{C})$ is nonvanishing, 
%$\chi$ is trivial, then we say that $\pi$
%has a nonzero $GL_2(F)$-period and $\pi$ is $GL_2(F)$-distinguished. 
we have the following results:
\begin{thm}
[\textbf{Main Theorem (Local)}]\label{localmain}
Let  $E$ be a quadratic field extension of a nonarchimedean local field $F,$
with the Galois group $\Gal(E/F)=\{1,\sigma \}$ and an associated quadratic character $\omega_{E/F}$ of $F^\times.$
Assume that $\pi$ is a generic irreducible smooth representation of $\GL_2(E),$ with central character $\omega_{\pi}=\mu^\sigma\mu^{-1},$ where $\mu$ is a character of $E^\times.$
\begin{enumerate}[(1)]
	\item
The following statements are equivalent:
\begin{enumerate}[(i)]
\item  $\pi$ is  $\GL_2(F)$-distinguished;
\item  $\pi=\BC(\tau)\otimes\mu^{-1}$ for some irreducible representation $\tau$ of $\GL_2(F)$ satisfying $\omega_\tau=\omega_{E/F}\cdot \mu|_{F^\times};$
\item the Langlands parameter $\phi_\pi:WD_E\rightarrow \GL_2(\mathbb{C})$ is conjugate-orthogonal in the sense of \cite[Section 3]{gan2011symplectic}. 
\end{enumerate}
 %Also, there is a counterpart for the inner form $D$ of $GL_2.$
%We also obtain a similar statement for the global case. 
\item Assume that $D$ is the nonsplit quaternion algebra defined over $F,$ then $D^\times(E)=\GL_2(E).$ Then the following statements are equivalent:
\begin{enumerate}[(i)]
	% If $\pi$ is an irreducible smooth admissible representation of $D(E)$ with central character $\omega_\pi,$ we assume $\omega_\pi|_{Z(D(F))}=1,$ then 
	\item $\pi$ is $D^\times(F)$-distinguished;% if and only if
\item $\pi$ is $\GL_2(F)$-distinguished
and $\pi\neq\pi(\chi_1,\chi_2),$ where $\chi_1\neq\chi_2$
and $\chi_1|_{F^\times}=\chi_2|_{F^\times}=1.$	%$\pi=BC(\tau)\otimes\mu^{-1}$ for some representation $\tau$ of $GL_2(F),$ where $\omega_\pi=\mu^\sigma/\mu$ and $\omega_\tau=\mu|_{F^\times}.$
 \end{enumerate}
\item Regarding $\pi|_{D^\times(F)}$ (resp. $\pi|_{\GL_2(F)}$) as a smooth representation of $PD^\times(F)$(resp. $\PGL_2(F)),$
then the following statements are equivalent:
\begin{enumerate}[(i)]
	\item $\dim \Hom_{\PGL_2(F)}(\pi|_{\GL_2(F)},\mathbb{C} )-\dim \Ext^1_{\PGL_2(F)}(\pi|_{\GL_2(F)},\mathbb{C} )=1  ;$
	\item $\dim \Hom_{PD^\times(F)}(\pi|_{D^\times(F)},\mathbb{C} )-\dim \Ext^1_{PD^\times(F)}(\pi|_{D^\times(F)},\mathbb{C} )=1  ;$
	\item $\pi$ is  $D^\times(F)$-distinguished.
\end{enumerate}
%in which case, the Langlands parameter $\pi_\pi$ is conjugate-symplectic.
\end{enumerate}
\end{thm}
For both $(1)$ and $(2),$ Flicker \cite[Page 162]{flicker1991distinguished} proved the cases for the principal series, and later Flicker and Hakim \cite[Theorem 0.1]{flicker1994quaternionic}  proved the cases for discrete series,  using relative trace formula. In this paper, we  use the local theta correspondence to give a new proof. There are three main ingredients for our proof:
\begin{itemize}
	\item the explicit big theta lifts $\Theta_{V,W,\psi}(\Sigma)$ and $\Theta_{W,V,\psi}(\tau )$;
	\item the Whittaker model $Wh_{N,\psi}(\Theta_{V,W,\psi }(\Sigma) )$ and
	\item the Whittaker model $Wh_{U,\psi_E}(\Theta_{W,V,\psi }(\tau)).$
\end{itemize}

 \par
If $F$ is a number field, let $\mathbb{A}=\mathbb{A}_F$ be the adele ring of $F.$
Let $G$ be a connected reductive group defined over $F$ and $H$
be a subgroup of $G.$ Assume that
$\pi$ is an  automorphic representation of $G(\mathbb{A}).$
Let  $Z(\mathbb{A} )$ be the center of $H(\mathbb{A}).$  If the integral
\[P(\phi)=\int_{H(F)Z(\mathbb{A})\backslash H(\mathbb{A})}{\phi(h)}dh\]
for $\phi\in\pi$
converges and is nonzero on $\pi,$ then we say that $\pi$ is $H$-distinguished.   Flicker \cite[Page 297]{flicker1988twisted} showed that a cuspidal automorphic representation $\pi$ of $\GL_n(\mathbb{A}_E),$ where $E$ is a quadratic extension of $F,$
 is $\GL_n(\mathbb{A})$-distinguished if and only if the partial Asai L-function $L^S(s,\pi,As )$ has a pole at $s=1.$ In this paper, we will use the  global theta lift and its Whittaker model to discuss the case
$G(\mathbb{A})=D^\times(\mathbb{A}_E)$ and $H(\mathbb{A})=D^\times(\mathbb{A}),$ where $D^\times$ is an inner form of $\GL_2$ defined over  $F.$
\par
%Assume $D$ is an inner form of $GL_2$ defined over $F$. Let $E$ be a quadratic extension of $F,$ with 
%Assume the Galois group $Gal(E/F)$ is generated by $\sigma.$ Let $\mathbb{A}_E=\mathbb{A}\otimes_F E$ be the adele ring of $E$.
%Set $G=D(\mathbb{A}_E ),~H=D(\mathbb{A} ).$ 
By the  Galois cohomology, the inner form $D^\times$ of $\GL_2$ corresponds to a quaternion algebra defined over $F$ with involution $\ast,$
denoted by $D.$   Then the tensor product $D\otimes_F E$ is an even Clifford algebra of $8$ dimensions defined over $F,$ denoted by $B.$ There are two natural $F$-linear automorphisms on $B,$ one is induced by the involution $\ast,$ and the other is induced by the Galois action.
More precisely, for the element $(d\otimes e)\in B,$ we define
\[(d\otimes e)^\ast=d^\ast\otimes e,\quad (d\otimes e)^\sigma=d\otimes\sigma(e). \]
Then $D$ is the set consisting of all Galois invariant elements and  $E$ coincides with the set consisting of all fixed points under the involution. Moreover, the intersection of
$D$ and $E$ coincides with $F.$
Set $X_D=\{x\in B|~x^\sigma=x^\ast \},$ which is a $4$-dimensional quadratic vector space defined over $F,$
with a quadratic form $q(x)=xx^\ast$ taking values in $E.$ In fact, if the element $q(x)$ is Galois invariant, then $q(x)$ lies inside $F$ and the quadratic form is well-defined on the space $X_D.$ Now we can define an $(F^\times\times B^\times)$-action on $X_D$ to be 
\[(\lambda,g).x=\lambda gx\sigma(g)^\ast, \]
where $\lambda\in F^\times,~g\in B^\times. $ 
In fact, the quaternion algebra  $D(E)$ over $E$ coincides with the even Clifford algebra $B.$
In \cite{roberts2001global}, Roberts  showed that the similitude group $\GSO(X_D,q)$ is a quotient of $F^\times\times D^\times(E),$ and  the special orthogonal group $\SO(Y,F)$ for some  $3$-dimensional subspace $Y$ in $X_D$ is $PD^\times(F).$  
\par
Given any cuspidal automorphic representation $\pi^D$ of $D^\times(\mathbb{A}_E),$
we may consider the Jacquet-Langlands correspondence representation $\pi=JL(\pi^D)$ of $\GL_2(\mathbb{A}_E ).$
Assume that the central character $\omega_\pi=\mu^\sigma\mu^{-1}$ for some Hecke character  $\mu$ of $\mathbb{A}_E^\times.$ Then $(\pi\otimes\mu)\boxtimes\mu|_{F^\times}$ is an automorphic representation of $\GSO(V,\mathbb{A}),$ where $V=F\oplus E\oplus F$ is a quadratic space over $F$ with a quadratic form $q(e,x,y)=N_{E/F}(e)-xy.$
%$\chi:F^\times\backslash\mathbb{A}^\times\rightarrow\mathbb{C}^\times,$ such that $\Sigma^D=\pi^D\boxtimes\chi$ is a cuspform on $GSO(X_D,\mathbb{A} ).$
%$GSO(X_D,\mathbb{A}).$ Let $$L(s,\Sigma^D,Std)=\prod_v L(s,\Sigma^D_v,Std)$$ be the complete L-function defined by Lapid and Rallis
%in \cite{lapid2005local}, which has a meromorphic continuity and satisfying a functional equation.
%There is a relation between  the L-function $L(s,\Sigma^D,Std )$ and the non-vanishing property of the theta lift in Yamana's paper \cite{yamana2014functions}.
%Let $S$ be a finite set of places containing all archimedean places of $F,$ such that
%for all $v\notin S,$ we have $E_v$ is unramified and $\pi_v$ is unramified.
%ramified palces over $E,$ and  % $D(F_v)$ ramified and $E_v$ is a field.  
%state the global result as follow.
\begin{thm}
[\textbf{Main Theorem (Global)}] Let $E$ be a quadratic extension over  a number field $F.$
Let $D^\times$ be an inner form of $\GL_2$ defined over  $F.$
Let $\pi^D$
be a cuspidal automorphic representation of $D^\times(\mathbb{A}_E),$ with $\omega_{\pi^D}|_{Z(\mathbb{A}) }=1.$
%where $(Res_{E/F}D')(F)\cong D(F)\otimes_F E.$ 
Assume that the Jacquet-Langlands
correspondence representation $\pi$  is a cuspidal automorphic representation
of $ \GL_2(\mathbb{A}_E).$ 
%Let $S$ to be a finite set of places containing all archimedean places in $F,$ such that $\forall v\notin S,$   $D_v$ splits and $E_v$ is unramified.
%Set $X_D=~\{x\in D\otimes_F E| x^\ast=\sigma(x) \},$ consider the theta lifting from
%$GSO(X_D)$ to $GL_2^+(W),$
 Then the following statements are equivalent:
\begin{enumerate}[(i)]\label{globalmain}
\item  $\pi^D$ is $D^\times(\mathbb{A})$-distinguished;
\item $\Sigma^D=(\pi^D\otimes\mu)\boxtimes\mu|_{\mathbb{A}^\times}$ as a cuspidal automorphic representation of $\GSO(X_D,\mathbb{A})$ is $PD^\times(\mathbb{A})$-distinguished for some Hecke character  $\mu:E^\times\backslash \mathbb{A}_E^\times\rightarrow\mathbb{C}^\times$ satisfying $\omega_{\pi^D}=\mu^\sigma\mu^{-1}$;
\item %The complete  $L$-function $L(s,\Sigma^D,Std)$
%has a pole at $s=1,$ %i.e. $L(1,\Sigma^D,Std)=\infty,$  
$\Sigma=(\pi\otimes\mu)\boxtimes\mu|_{\mathbb{A}^\times}$ as a cuspidal automorphic representation of $\GSO(V,\mathbb{A} )$ is $\PGL_2(\mathbb{A} )$- distinguished and
for each local place $v$ of $F$ where $D_v$ is ramified and $E_v$ is a field, the local representation $\pi_v$ is isomorphic to %either supercuspidal, or a Steinberg representation or a principal series of the form
$\BC(\tau_v)\otimes\mu_v^{-1}$ for some representation %,\chi_2)$ with $\chi_1\chi_2^\sigma=1,$
% and $\chi_1|_{F_v^\times}=\chi_2|_{F_v^\times}=1,$ 
%for all places $v$ when $D_v$ ramified and $E_v$ is a field, 
 $\tau_v$ of $\GL_2(F_v)$ with $\omega_{\tau_v}=\mu_v|_{F^\times_v}\omega_{E_v/F_v}$ and $\tau_v|_{\GL_2^+}$ is irreducible, where $\GL_2^+=\{g\in \GL_2(F_v)|\omega_{E_v/F_v}(\det(g))=1 \}.$
%$\tau_v:E^\times_v\rightarrow\mathbb{C}^\times,i=1,2 $ are characters.
%if $\Sigma_v $ is an infinitely-dimensional representation, 
%$\pi^D_v =\mu_v^c/\mu_v$  if 
%$\Sigma_v$ is an one-dimensional representation.
\end{enumerate} 
\end{thm}
\par
In a short summary, we give a criterion to  the local distinction problem of $\GL_2$, which is well-known in \cite[Theorem 0.2]{flicker1994quaternionic}. But we use a different method, i.e., the theta correspondence and base change to recover the local period problems for both $\GL_2$ and $D^\times.$ Then  the Ext version for the Branching law can be obtained easily via the Mackey theory.
%The proof of the global case highly relies on the result in \cite[Proposition 5.2]{ganshimura}, which transferes the period problem to another one whether the global theta lifting is generic or not. But that an automorphic form $\tau\cong\otimes_v\tau_v$ of $GL_2(\mathbb{A})$ is generic or not depends on each local situation $\tau_v$, hence we only need to figure out whether $\theta(\Sigma_v^D)$ is generic or not, at each local place v.  And we use the similar trick to deal with the situation for $GL^+_2.$
%We also recover the result  \cite[Theorem 3]{flicker1994quaternionic} about the relation of global periods problems between $GL_2$ and its inner form. 
In the global situation, we use the  regularized Siegel-Weil formula in the first term range, which appears in \cite[Corollary 7.9]{gan2011regularized}, and the global theta lifts from orthogonal groups to show \cite[Thereorem 0.2]{flicker1994quaternionic}, i.e. our main global theorem.
However, our method can not be extended to the cases for the general inner forms of $\GL_n$ in \cite[Theorem 0.5]{flicker1994quaternionic}.

Now we briefly describe the contents and the organization of this paper.
In section $2$, we set up the notation about the quadratic vector spaces. In section \ref{sec3}, the explicit local theta lifts are given, including the small theta lifts and the big theta lifts.
 In section $4$, we give the proof for the \textbf{Main Theorem (Local)}.
 In section $5,$  the global theta lifts are introduced. In section $6,$ we give the proof of  the \textbf{Main Theorem (Global)}.
  %compute the Whittaker coefficients of the theta lift and  prove the local period problems for $GL_2.$ In section 5, we will deal with the global inner form period problems.

%% file: preliminaries.tex
\section{Preliminaries}

We follow the notation in \cite{roberts2001global}
to introduce the $4$-dimensional quadratic spaces. Let $F$ be a field with characteristic not equal to $2,$ and let $E$ be a quadratic extension over $F$ with Galois group $\Gal(E/F)=\{1,\sigma\}.$
Assume that $D$ is a non-split quaternion algebra defined over $F$ with an involution $^\ast.$ Then
$B=D\otimes_FE$ is an  algebra over $F,$ with Galois action $\sigma$ and involution
$^\ast.$ 
Set $X_D=\{x\in B|~x^\ast=x^\sigma \}$ to be a $4-$dimensional vector space
 with a quadratic form $q$ taking values in $F.$ 
\par
Now we define an $(F^\times\times B^\times)-$action on $X_D$, for $\lambda\in F^\times$ and $g\in B^\times, $ define
\[(\lambda,g)x=\lambda gx\sigma(g)^*. \]
This action preserves $X_D,$ and $q((\lambda,g)x)=\lambda^2N_{E/F}(gg^\ast)\cdot q(x).$
Thus, the element $(\lambda,g)$ lies inside the similitude orthogonal group $\GO(X_D,F),$ with similitude factor $\lambda^2N_{E/F}(gg^\ast).$
In fact, it lies in the  similitude special orthogonal group $\GSO(X_D,F).$
\begin{thm}\cite[Theorem 2.3]{roberts2001global}\label{similitude}
	There is an exact sequence 
	\[
	\xymatrix{1\ar[r]& E^\times \ar[r] &F^\times\times B^\times \ar[r]& GSO(X_D,F)\ar[r]&1,}
	\] where the first arrow is given by $e\mapsto (N_{E/F}(e^{-1}),e).$
\end{thm}
Let us consider the vector $1\otimes 1\in X_D,$ denoted by $y.$ For arbitrary $d\in D^\times,$ the element
$(N_{D/F}(d^{-1}),d)\in F^\times\times B^\times$ fixes the vector $y.$ Set $y^\perp$ to be the orthogonal subspace with respect to $y$ in the quadratic space $X_D.$
\begin{thm}\cite[Proposition 2.6]{roberts2001global}\label{SO(3)}
	Assume that $D$ is a quaternion algebra over $F,$ let $y,X_D,B$ be the notation as above.
	% in B corresponding to y. 
	Then there exist an exact sequence
	\[1\rightarrow F^\times\rightarrow D^\times \rightarrow \SO(y^\perp,F)\rightarrow1, \]
	and the following commutative diagram
	\[\xymatrix{1\ar[r]& F^\times\ar[r]\ar[d]&\ar[r]D^\times \ar[r]\ar[d]&\SO(y^\perp,F)\ar[r]\ar[d]&1\\
		1\ar[r]&E^\times\ar[r]& F^\times\times B^\times\ar[r]& \GSO(X_D,F)\ar[r]&1,
	}\]
	where the inclusion from $D^\times$ to $F^\times\times B^\times $ is given by $g\mapsto (N_{D/F}(g^{-1}),g).$
\end{thm}

Assume that $X$(resp. $Y$) is a $4$-dimensional quadratic space over $F$ with a quadratic form $q_X$(resp. $q_Y$), pick one anisotropic vector $v_0\in X$ with $q_X(v_0)\neq0.$ Assume that $L$ is an anisotropic line in $Y,$
we define that a pair $(Y,L)$ is isomorphic to the pair $(X,Fv_0)$ if there is an invertible linear transform $T:X\rightarrow Y$ such that $q_Y\circ T=\lambda q_X$ for some $\lambda\in F^\times$ and $T(v_0)\in L.$
\begin{lem}\label{DE}
	There are natural $1-1$ correspondences between the following sets:
	\begin{enumerate}[(i)]
		\item the pair $(D,E),$ where $D$ is  a quaternion algebra  defined over  $F,$ and $E$ is a separable quadratic extension algebra over $F;$
		\item the $3$-dimensional quadratic space class $W$ over $F$;
		\item the $4$-dimensional quadratic space class $(X,q_X)$ with a fixed anisotropic vector $v_0$ such that $q(v_0)=1;$
		\item the triple class $(Y,q_Y,L)$ where $L$ is an anisotropic line in $Y.$
	\end{enumerate} 
	%	Moreover, we have a group isomorphism $SO(v_0^\perp)\cong PD^\times.$
\end{lem}
\begin{proof}
	$(i)\Leftrightarrow(ii)$ comes from the Galois cohomology. Note that $\Oo(3)=\SO(3)\times\mu_2$. Then
	\[H^1(F,\Oo(3))\cong H^1(F,\mu_2)\times H^1(F,\PGL_2)\cong F^\times/{F^\times}^2\times Br_2(F) .\]
	The Brauer group  $Br_2(F)$ corresponds to quaternion division algebras over $F,$ and the quotient group $F^\times/{F^\times}^2$ corresponds to the separable quadratic algebras over $F.$
	\par
	$(ii)\Leftrightarrow (iii)$ is easy to see if we choose $W$ to be the subspace $v_0^\perp.$
	\par
	%$W=\{v\in X:q(v+v_0)=q(v)+q(v_0) \}.$\\
	$(iii)$ is equivalent to $(iv)$ by definition.
\end{proof}
Here we have a group isomorphism $\SO(W)\cong PD^\times$  if $W =v_0^\perp\subset X_D. $
\begin{rem}
In fact, $q(v_0)\in F^\times/(F^\times)^2$ can be any nonzero element in the number field $F.$
%In the local case, assume $\mathfrak{v}$ is a finite place of F, $\mathfrak{b} $ is a place of E lying over $\mathfrak{v}.$
%Fix the Etale extension $E/F.$
\end{rem}
Let $F$ be a number field, and let $E$ be a quadratic extension of $F.$ Let $D$ be a quaternion algebra defined over $F.$
Let $S_{D,E}$ be the set of places $v$ of $F$ such that $D_v$ is ramified and $v$ splits in $E.$
\begin{prop}
	Assume that  $D$ and $D'$ are quaternion algebras over $F.$ Then $D\otimes E\cong D'\otimes E$ as $E-$algebras if and only if $S_{D,E}=S_{D',E}.$
\end{prop}

%% file: local.tex
\section{Local Theta Correspondence}\label{sec3}
Now we briefly recall some facts about the local theta correspondences for similitude groups, referring to  \cite{moeglin1987correspondences}, \cite{kudla1996notes},  \cite{roberts2001global} and \cite{gan2011theta}.
\par
Let $F$ be a local field of characteristic zero.
Consider the dual pair $\Oo(V)\times \Sp(W).$
For simplicity, we may assume that $\dim V$ is even . Fix a nontrivial additive character $\psi$ of $F.$
Let $\omega_\psi$ be the Weil representation for $\Oo(V)\times \Sp(W).$
If $\pi$ is an irreducible representation of $\Oo(V)$ (resp. $\Sp(W)$), the maximal $\pi-$isotypic quotient has the form 
\[\pi\boxtimes\Theta_\psi(\pi) \]
for some smooth representation of $\Sp(W)$ (resp. $\Oo(V)$). We call $\Theta_\psi(\pi )$ or $\Theta_{V,W,\psi}(\pi)$
the big theta lift of $\pi.$ It is known that $\Theta_\psi(\pi)$ is of finite length and hence is admissible. Let $\theta_\psi(\pi)$ or $\theta_{V,W,\psi}(\pi)$ be the maximal semisimple quotient of $\Theta_\psi(\pi),$ which is called the small theta lift of $\pi.$ Then there is a conjecture of Howe which states that
\begin{itemize}
	\item $\theta_\psi(\pi)$ is irreducible whenever $\Theta_\psi(\pi)$ is non-zero.
	\item the map $\pi\mapsto \theta_\psi(\pi)$ is injective on its domain.
\end{itemize}
This has been proved by Waldspurger when the residual characteristic $p$ of $F$ is is not $2.$ Recently, it has been proved completely, see \cite{gan2014howe},\cite{gan2014proof} and  \cite{gan2015howe}.
\par
Then we 
 extend Weil representation to the case of similitude groups. Let $\lambda_V$ and $\lambda_W$
be the similitude factors of $\GO(V)$ and $\GSp(W)$ respectively. 
We shall consider the group
\[R=\GO(V)\times \GSp^+(W) \]
where $\GSp^+(W)$ is the subgroup of $\GSp(W)$ consisting of elements $g$ such that $\lambda_W(g)$
lies in the image of $\lambda_V.$ Define
\[R_0=\{(h,g)\in R|~\lambda_V(h)\lambda_W(g) =1\} \]
to be the subgroup of $R.$ The Weil representation $\omega_\psi$ extends naturally to the group $R_0$
via \[\omega_\psi(g,h)\phi=|\lambda_V(h)|^{-\frac{1}{8}\dim V\cdot\dim W } \omega(g_1,1)(\phi\circ h^{-1}) \]
where \[g_1=g\begin{pmatrix}
\lambda_W(g)^{-1}&0\\0&1
\end{pmatrix}\in Sp(W). \]
Here the central elements $(t,t^{-1})\in R_0$ acts by the quadratic character $\chi_V(t)^{\frac{\dim W}{2}},$
which is slightly different from the normalization used in \cite{roberts2001global}.
\par
Now we consider the compactly induced representation
\[\Omega=ind_{R_0}^R\omega_\psi. \]
As  a representation of $R,$ $\Omega$ depends only on the orbit of $\psi$ under the evident action of $Im\lambda_V\subset F^\times.$ For example, if $\lambda_V$ is surjective, then $\Omega$ is independent of $\psi.$ For any irreducible representation $\pi$ or $\GO(V)$ (resp. $\GSp^+(W)$),
the maximal $\pi-$isotropic quotient of $\Omega$ has the form
\[\pi\otimes\Theta_\psi(\pi) \]
where $\Theta_\psi(\pi)$ is some smooth representation of $\GSp^+(W)$ (resp. $\GO(V)$). Similarly, we let
$\theta_\psi(\pi)$ be the maximal semisimple quotient of $\Theta_\psi(\pi).$
Note that though $\Theta_\psi(\pi)$ may be reducible, it has a central character $\omega_{\Theta_\psi(\pi)}$
given by
 \[\omega_{\Theta_\psi(\pi) }=\chi_V^{\frac{\dim W}{2}}\omega_\pi . \] 
 There is an extended Howe conjecture for similitude groups, which says that $\theta_\psi(\pi)$ is irreducible whenever $\Theta_\psi(\pi)$
is non-zero and the map $\pi\mapsto\theta_\psi(\pi)$ is injective on its domain.
It was shown by Roberts \cite{roberts1996theta} that this follows from the Howe conjecture for isometry groups. 
\par In our cases,
 there are two types of $4$-dimension quadratic spaces over $F$ with discriminant  $E.$ Let $\epsilon\in F^\times$ such that $\omega_{E/F}(\epsilon)=-1.$
 Set 
 \[\begin{cases}
 V^+=(V_E,N_{E/F})\oplus\mathbb{H}\\
 V^-=\epsilon V^+
 \end{cases} \] 
 where $V_E=E$ as a $2$-dimensional vector space over $F$. Let $V=V^+$ with a quadratic form $q,$  we write elements of $V$ as $$v=\begin{pmatrix}
 a&x\\\sigma(x)&b
 \end{pmatrix},~\mbox{ and the quadratic form}~ q(v)=N_{E/F}(x)-ab.$$
 The quadratic character  $\chi_V$ is $\omega_{E/F}$ and the Hasse-invariant $$\epsilon(V)=1.$$
  For any vector $v=\begin{pmatrix}
  a&x\\\sigma(x)&b
  \end{pmatrix}\in V^-,$ we define a quadratic form $$q^-(v)=\epsilon(N_{E/F}(x)-ab).$$
  The quadratic character $\chi_{V^-}$ is $\omega_{E/F}$ and the Hasse-invariant of $V^-$ is $-1.$
  \par
  By Theorem \ref{similitude},  we have an isomorphism
   \[\GSO(V,q)\cong \frac{\GL_2(E)\times F^\times}{\triangle E^\times },~\mbox{where}~\triangle(t)=(t,N_{E/F}(t^{-1})),t\in E^\times. \]
   Later, we denote $\GSO(V^+)$ as $\GSO(3,1),$ and denote
   $\GSO(V^-)$ as $\GSO(1,3).$
   \par
    \subsection*{Representations of $\GSO(V)$} Assume that $\chi$ is a character of $F^\times.$
    Given a representation $\pi$ of $\GL_2(E),$ then the representation $\pi\otimes\chi$ of $\GL_2(E)\times F^\times$ is a representation of $\GSO(V)$
    if and only if $\omega_\pi=\chi\circ N_{E/F}.$ If so, we  denote $\pi\boxtimes\chi$ as the representation of $\GSO(V).$ Fix a parabolic subgroup $P,$ which stabilizes the isotropic line $\Bigg\{\begin{pmatrix}
    x&0\\0&0
    \end{pmatrix} \Bigg\}$ in $V,$ assume that $P=MN$ and the Levi subgroup $M$ is a quotient of $T(E)\times F^\times,$
    where $T(E)$ is the diagonal torus of $\GL_2(E).$
    % In fact, $M$ is isomorphic to $F^\times\times E^\times.$
    Assume that $\mu:E^\times\rightarrow\mathbb{C}^\times$ is a character, denote $V_E=E$ as a $2-$dimensional quadratic space over $F.$ 
    %Then, if $\mu$ is Galois invariant, there are two extensions $\mu^{\pm}$ to $GO(V_E),$
    % whereas if $\mu$ is not Galois invariant, then we set $\mu^+=\mu^-=ind_{GSO(V_E)}^{GO(V_E)}\mu.$
    Consider the normalized induced representation $I_P(\mu,\chi)$ of $\GSO(V).$
    \begin{lem}
    	\cite[Lemma A.6]{gan2011endoscopy} For a character $\mu$ of $E^\times,$ we have
    	\[I_P(\mu,\chi)=\pi((\chi\circ N_{E/F})\mu,\mu^\sigma )\boxtimes(\chi\cdot\mu|_{F^\times}) .\]
    \end{lem}
  In fact, we shall only consider the theta correspondence for $\GSO(V)\times \GSp^+(W).$
    There is no significant loss in restricting to $\GSO(V)$ because of the following lemma.
    \begin{lem} Assume that $\dim W=2$ and $\dim V=4.$
    	Let $\pi$ (resp. $\tau$) be an irreducible representation of $\GSp^+(W)$ (resp. $\GO(V)$) and suppose that
    	\[\Hom_R(\Omega,\pi\otimes\tau )\neq0. \]
    	Then the restriction of $\tau$ to $\GSO(V)$ is irreducible. If $\nu_0=\lambda_V^{-2}\cdot\det$ is the unique nontrivial quadratic character of $\GO(V)/\GSO(V),$ then $\tau\otimes\nu_0$ does not participate in the theta correspondence with $\GSp^+(W).$
    \end{lem}
    One may just follow the proof of Gan and Takeda in \cite[Lemma 2.4]{gan2011theta} to obtain it.
    \begin{prop}\cite[Proposition 2.3]{gan2011theta}\label{thetaofsup}
    	Suppose that $\pi$ is a supercuspidal representation of $\GO(V)$ (resp. $\GSp^+(W)$). Then 
    	$\Theta_\psi(\pi)$ is either zero or  an irreducible representation of $\GSp^+(W)$ (resp. $\GO(V)$).
    \end{prop}
   
%    \begin{proof} Note that $\tau$ is irreducible when restricted to $GSO(V)$ if and only if $\tau\otimes\nu_0\neq\tau.$ Consider $\tau|_{O(V)}=\oplus_i\tau_i$ by \cite{roberts1996theta}, and
%    	$\tau_i|_{SO(V)}$ is irreducible and $\tau_i\otimes\nu_0\neq\tau_i$ by Rallis's result, which can be found in \cite[Section 5, Pg 282]{prasad1993local}. %Hence $\tau\otimes\nu_0\neq\tau,$ i.e. $\tau|_{GSO(V)}$
%   %is irreducible. 
%   And	if $\tau$ participates in the theta correspondence with $GSp^+(W),$ we know $\tau_i\otimes\nu_0$ does not participate in the theta correspondence with $Sp(W)$ by Rallis's result as well.
%   This implies that $\tau\otimes\nu_0\neq\tau$ and $\tau\otimes\nu_0$ does not participate in the theta correspondence with $GSp^+(W).$
%   %then the first occurrence index is less than $2.$ By the conservation relation in the similitude version,
%    %	the first occurrence index of $\tau\otimes\nu_0$ is bigger than $2,$ i.e. at least $3,$ which means that
%    	%the theta lift of $\tau\otimes\nu_0$ from $GO(V)$ to $GSp^+(W)$ must be zero.
%    \end{proof}

 \subsection{Whittaker model}
  Given an irreducible representation $\Sigma$ of $\GSO(V),$ we consider the Whittaker model of the big theta lifts $\Theta_\psi(\Sigma).$ 
  \par
  Recall that, for $n(x)=\begin{pmatrix}
 1&x\\0&1
 \end{pmatrix}\in \GSp^+(W)$ and $\phi\in S(NE^\times\times V),$ we have
 $$n(x)\phi(t,v)=\psi(tx\cdot q(v))\phi(t,v).$$
Assume that $\Omega=S(NE^\times)\otimes S(V)$ is the Schwartz function space. Given $a\in F^\times,$ set \[ \Omega_{N,\psi_a}=\Omega/<n(x)\phi-\psi(ax)\phi>.
 \]
 Given an irreducible representation $\Sigma\cong \pi\boxtimes\chi$ of $\GSO(V),$ consider the set $$V_a=\{(t,v)\in NE^\times\times V|~tq(v)=a \}.$$ Pick a point $y_a=(1,a)\in V_a,$
 where $a=\begin{pmatrix}
 a&0\\0&-1
 \end{pmatrix}\in V.$ Let $Z$ be the center of $\GL(W).$  Since $\GSO(V)\times Z$ acts on $V_a$ transitively with
 the action $$(h,b)(t,v)=(\lambda_V(h)\cdot b^2t,bh^{-1}v),$$
 the point $y_a$ has the stabilizer $$J_a=\{(h,b)\in (\GSO(V)\times Z)\cap R_0|~hy_a=by_a \}\subset P_a\times Z,$$
 where $P_a$ is the closed subgroup of $\GSO(V)$ fixing the anisotropic line $Fy_a$ in $V.$
 \begin{prop}\label{whittaker}Given an irreducible smooth representation $\Sigma$ of $\GSO(V),$ we have 
 	$$\Hom_N(\Theta_\psi(\Sigma),\psi_a)\cong \Hom_{\GSO(V)}(\Omega_{N,\psi_a},\Sigma)\cong
	\Hom_{\SO(y_a^\perp)}(\Sigma^\vee,\mathbb{C}).$$
 	\end{prop} 
It is very similar to \cite[Proposition 5.5]{ganshimura}.

 	\subsection*{Mixed model}
 	Now we introduce the mixed model to deal with the theta lift from $\GSO(V)$ to $\GSp^+(W).$
 	From the Witt decomposition $W=X+Y$ and
 	\[V=Fv_0+V_E+Fv_0^\ast \quad\Big(v_0=\begin{pmatrix}
 	1&0\\0&0
 	\end{pmatrix}~\mbox{and}~v_0^\ast=\begin{pmatrix}
 	0&0\\0&1
 	\end{pmatrix}\Big),
 	\]
 	 we obtain $$V\otimes X+V\otimes Y=(Fv_0\otimes W )+ V_E\otimes(X+Y)+(Fv_0^\ast\otimes W). $$
 	Note that the two isotropic subspaces $Fv_0\otimes X$ and $Fv_0^\ast\otimes Y$ are paired via the natural symplectic form $\langle-,-\rangle$ on $V\otimes W.$ The intertwining map
 	\[I:S(V\otimes Y)\rightarrow S(V_E\otimes Y)\otimes S(W\otimes v_0^\ast ) \]
 	is given by a partial Fourier transform: for $v\in V_E\otimes Y,$ $x\in v_0^\ast\otimes X,~y\in v_0^\ast \otimes Y,$
 	and $z\in v_0\otimes Y,$ we have
 	\[(I\varphi)(v,x,y)=\int_F\psi(zx)\varphi\Big(\begin{pmatrix}
 	z\\v\\y
 	\end{pmatrix} \Big)dz. \]
 	Let $Q$ be the maximal parabolic subgroup of $SO(V)$ which stabilizes $Fv_0,$ and let U be its unipotent radical. Then for $h=\begin{pmatrix}
 	1&b\\0&1
 	\end{pmatrix}\in U,$ we have
 	\[I(\omega_\psi(h)\varphi)(v,x,y)= \psi( x\cdot tr_{E/F}(b\sigma(v)))\psi(-N(b)xy)(I\varphi)(v,x,y) .\]
 	For an element $m$ in the Levi subgroup of $Q,$ set $m=\Big(\begin{pmatrix}
 	1&\\&d
 	\end{pmatrix},\lambda \Big),$ we have
 	\[I(\omega_\psi(m)\varphi )(v,x,y)=|\lambda|(I\varphi)\Big(\frac{v}{\lambda \sigma(d)},\frac{x}{\lambda N(d)},\lambda y\Big) .\]
 	For $g\in Sp(W),$ regard $(I\varphi)(v,x,y)=(I\varphi)(x,y)(v)$ as a function defined on $S( v_0^\ast\otimes W )$ taking values in $S(V_E\otimes Y)$. Then we have
 	\[I(\omega_\psi(g)\varphi)(v,x,y)=\Bigg(\omega_0(g)(I\varphi)\Big(g^{-1}\begin{pmatrix}
 	x\\y
 	\end{pmatrix}\Big)\Bigg)(v) \]
 	where $\omega_0$ is the Weil representation of $\Sp(W)\times \SO(V_E).$
 	\par
 	Now we extend the Weil representation $\omega_\psi$ on $\SO(V)\times \Sp(W) $ to the  group $R_0$ by
 \[\omega_\psi(g,h)\phi=|\lambda_V(h)|^{-\frac{1}{8}\dim V\cdot\dim W } \omega(g_1,1)(\phi\circ h^{-1}) \]
 where \[g_1=g\begin{pmatrix}
 \lambda_W(g)^{-1}&0\\0&1
 \end{pmatrix}\in Sp(W). \]
 And we consider the compact induction
 	\[\Omega=ind_{R_0}^R\omega_\psi. \]
 	
 	Set $\psi_E=\psi\circ tr_{E/F}$ to be an additive character of $E.$ The unipotent radical subgroup $U$  of $\GSO(V)$ is isomorphic to $E.$ 
 	\begin{prop}\label{thetamodeltran}
 		Assuming that $\sigma$ is an irreducible representation of $\GSp^+(W),$ then
 		\[\Hom_U(\Theta_\psi(\sigma),\psi_E)\cong \Hom_N(\sigma,\psi). \]
 	\end{prop}
This follows from \cite[Proposition 6.3]{ganshimura}.
 	\begin{coro}\label{generic}
 		Assume that $\sigma$ is an irreducible representation of $\GSp^+(W)$. Then
 		$\sigma$ is $\psi$-generic if and only if the big theta lift $\Theta_\psi(\sigma)$ is generic. Here we use the fact that all Whittaker data are equivalent for $\GSO(3,1).$
 	\end{coro}
 	%The reason is that if a representation $\pi$ of $GL_2(E)$ is generic with respect to $\psi_E,$
 	%then $\pi$ is generic for any character $.$
  \subsection{Principal series} Now we start to consider the explicit local theta lifts.
  \par
   Assume that $B$ is a Borel subgroup of $\GSp(W)$ with a Levi subgroup $T.$ We define a subgroup $B^+=B\cap \GSp^+(W)$ and its torus subgroup $T^+=T\cap B^+.$  
   
 Let us start from the irreducible principal series representations  of $\GSp(W)$. Set $$\GL_2^+=\GL_2^+(F)=\GSp^+(W)=\{g\in \GL_2(F)|\omega_{E/F}(g)=1 \}.$$
%\subsection{Principal series}
%\paragraph*{Theta lift from $GL^+_2(W)$}
%First, let us look at the representations of $GL_2^+(F).$
\begin{lem}
Assuming that $\tau$ is an irreducible infinitely dimensional representation of $\GSp(W),$ then
$\tau|_{\GL_2^+}$ is reducible if and only if $\tau\otimes\omega_{E/F}\cong \tau ,$
in which case, we call it dihedral with respect to $E$.
\end{lem}
%\begin{proof} Let
%$\omega_{E/F}:GSp(W)\rightarrow\mathbb{C}^\times$ be a character of $GSp(W),$ set  $\omega_{E/F}(g)=\omega(\det(g)),$ then $\ker\omega_{E/F}= GL_2^+.$
% And
%$$Hom_{GL_2^+}(\tau|_{GL_2^+},~\tau|_{GL_2^+})\cong Hom_{GL_2(F)}(\tau,~Ind(\tau|_{GL_2^+}))\cong Hom_{GL_2(F)}(\tau,~\tau\oplus (\tau\otimes\omega_{E/F})). $$
%Its dimension is two if and only if $\tau\cong\tau\otimes\omega_{E/F}.$
%\end{proof}
%\begin{rem}
%	If $\dim_F W=2,$ sometimes we write $GSp^+(W)=GL_2^+(F)$ as $GL_2^+.$
%\end{rem}
\begin{prop}
Assume that $\tau=\pi(\chi_1,\chi_2)$ is an irreducible principal series representation of $\GSp(W).$ Then
$\chi_1\neq\chi_2\cdot|-|^{\pm1}.$
%$\tau|_{GL_2^+}$ may be reducible.
\begin{enumerate}[(i)]\label{symW}
	\item If $\tau\ncong \tau\otimes\omega_{E/F}$ and $\chi_1\neq\chi_2\omega_{E/F}|-|^{\pm1} .$ Then $\tau|_{\GL_2^+}$ is irreducible and $$\Theta_\psi(\tau^+)= \theta_\psi(\tau^+)\cong\pi(\chi_2\circ N_{E/F},\chi_1\circ N_{E/F})\boxtimes \chi_1\chi_2\chi_V.$$%~\mbox{denoted by} ~\pi(\chi_2\circ N,\chi_1\circ N )\boxtimes\chi_1\chi_2\omega_V $$
	\item If $\tau=\pi(\chi_3,\chi_3\omega_{E/F})$ is dihedral with respect to $E$, then
	$\tau|_{\GL_2^+}\cong \tau^+\oplus\tau^-$
	and $$\Theta_\psi(\tau^+)= \theta_\psi(\tau^+)\cong \pi(\chi_3\circ N_{E/F},\chi_3\circ N_{E/F})\boxtimes\omega_{\tau^+}\omega_{E/F},~\mbox{while the other one}~\theta_\psi(\tau^-)=0,$$
	where $\omega_{\tau^+}$ is the central character of the representation $\tau^{+}.$
	\item If $\tau=\pi(|-|^{1/2},|-|^{-1/2}\omega_{E/F})\otimes\chi_4$ or $\pi(|-|^{-1/2},|-|^{1/2}\omega_{E/F})\otimes \chi_4,$ 
	then $\tau|_{\GL_2^+}$ is irreducible and
	$$\Theta_\psi(\tau|_{GL_2^+} )\cong \big(\pi(|-|_E^{1/2},|-|_E^{-1/2} )\otimes\chi_4\circ N_{E/F}\big)\boxtimes\chi_4\circ N_{E/F}~ \mbox{and}~ \theta_\psi(\tau|_{\GL_2^+})\cong \chi_4\circ N_{E/F}\boxtimes\chi_4\circ N_{E/F},$$
	where  the absolute valuation $|-|_E$ on $E$ is given by $|e|_E=|N_{E/F}(e)|,~\forall e\in E.$
	%where $St_E$ is the Steinberg representation of $GL_2(E).$
\end{enumerate} 
\end{prop}

%
%\textbf{Claim 4}: $Hom_M(ind_{Q}^{GSp^+(W)\times M}\mathbb{C},\delta^{1/2}\chi_M)\cong  I_{B^+}(||(\chi_1/\chi_2)^{-1},||^{-1/2}\chi_2^{-1})^\vee.$

%% a&\\&b
%% \end{pmatrix}f)(t,x)=\int_{NE^\times\times F^\times} |b|^{-1}\chi_V(b)f((abt,bx)m)|\lambda N(d)|(\delta^{1/2}\chi_M)^{-1}(m)dm=|b|^{-1/2}\chi_V(b)|a|^{-3/2}\chi_1(b)\chi_2(a)\omega_{E/F}(b)l(f)(t,x)$

%% file: squareintegral3.tex
%In fact, we have seen $\theta(St_F|_{GL_2^+})=St_E\boxtimes \omega_{E/F}$ in \autoref{StF}.
Let us turn the tables around. Assume that $\Sigma=\pi\boxtimes\chi$ is an irreducible representation of $\GSO(V),$ where $\omega_\pi=\chi\circ N_{E/F}.$ We only consider the big theta lift from $\GSO(V)$
to $\GSp^+(W)$ due to the Howe duality.
\begin{prop}
	Assume that $\pi$ is an irreducible principal series representation of $\GL_2(E)$. Set $\Sigma=\pi\boxtimes\chi$.
	Then $\Theta(\Sigma)=0$ if $\pi\neq\pi^\sigma.$   If $\pi=\BC(\tau),$ then
	\begin{enumerate}[(i)]\label{orthV}
				\item If $\tau=\pi(\chi_1,\chi_2)$ is a principal series and $\chi_1\chi_2^{-1}$ is neither $\mathbf{1}$ nor $\omega_{E/F}$, then
				\begin{enumerate}[(a)]
					\item $\Theta_\psi(\pi\boxtimes\chi_1\chi_2\omega_{E/F})=\pi(\chi_2,\chi_1)|_{\GL_2^+} $ and
					\item $\Theta_\psi(\pi\boxtimes\chi_1\chi_2 )=\pi(\chi_2,\chi_1\omega_{E/F} )|_{\GL_2^+}.$
				\end{enumerate}
						\item If $\tau$ is a dihedral principal series, i.e. $\pi=\pi(\chi_3\circ N_{E/F},\chi_3\circ N_{E/F}),$
						where $\chi_3$ is a character of $F^\times,$
						we consider the representations $\Sigma_1=\pi\boxtimes\chi_3^2\omega_{E/F} $ and $\Sigma_2=\pi\boxtimes\chi_3^2,$ then
						\begin{enumerate}[(a)]
						\item		 $\Theta_\psi(\Sigma_1)=\pi(\chi_3,\chi_3)|_{\GL_2^+}$ and
					\item		$ \Theta_\psi(\Sigma_2)=\Ext^1_{\GL_2^+}(\tau^-,\tau^+)$ is reducible, where $\pi(\chi_3,\chi_3\omega_{E/F})|_{\GL_2^+}=\tau^+\oplus\tau^-$ and $\tau^+$ (resp. $\tau^-$ ) is the $\psi-$generic (resp. $\psi_\epsilon-$generic) component of $\pi(\chi_3,\chi_3\omega_{E/F})|_{\GL_2^+}$.
						\end{enumerate}
		\item If $\tau$ is dihedral supercuspidal with respect to $E$, i.e. $\pi=\BC(\tau)=\pi(\chi_4,\chi_4^\sigma)$ and $\chi_4\neq\chi_4^\sigma,$ then the theta lift 
		$\theta_\psi(\pi\boxtimes\omega_\tau)$ is zero and $$\Theta_\psi(\pi\boxtimes\omega_\tau\omega_{E/F})=\theta_\psi(\pi\boxtimes\omega_\tau\omega_{E/F})=\tau^+,$$
		where $\tau^+$ is the $\psi-$generic component of $\tau|_{\GL_2^+}.$
			\end{enumerate}
\end{prop}
We follow the method in \cite[Appendix A]{gan2011theta} to compute the Jacquet modules.
For the full details of computation, one may refer to \cite[Section 2.4]{lu2016} 
 
 We would like to highlight the fact that given a tempered representation $\tau^+=\theta_{V_E,W,\psi}(\mathbf{1})$ of $\Sp(W),$ which is the theta lift of the trivial representation of $\SO(V_E)\cong E^1,$  the big theta lift $\Sigma=\Theta_{W,V,\psi}(\tau)$ of $\tau$ from $\Sp(W)$ to $\SO(V)$ is irreducible due to \cite[Proposition C.4(i)]{gan2014formal}.  However, $\Theta_{V,W,\psi }(\Sigma)$ is a reducible representation of $\Sp(W).$
 In fact, $\Theta_{V,W,\psi }(\Sigma)=\Ext^1_{\Sp(W)}(\tau^-,\tau^+)$
 where $\tau^+\oplus\tau^-=I(\omega_{E/F})$ as a reducible principle series of $\Sp(W)$ and $\tau^-=\theta_{\epsilon V_E,W,\psi}(\mathbf{1}).$ This is why we can get Proposition \ref{orthV}(ii)(b).

%% file: steinberg2.tex
\subsection{Steinberg Representations}
Now we consider the Steinberg representation which is the unique quotient of a reducible principal series representation.
\begin{thm}
\begin{enumerate}
\item Let $St_F$ be the Steinberg representation of $\GL_2(F)$.
Then $St=St_F|_{\GL_2^+(F)}$ is irreducible and $\Theta_\psi(St)=\theta_\psi(St)\cong \BC(St_F)\boxtimes\omega_{E/F}.$
\label{StF}
\item Let $St_\chi=St_F\otimes\chi$ be the twisted Steinberg representation of $\GL_2(F).$ 
Then $\BC(St_\chi)$ is an irreducible representation of $\GL_2(E)$, and $\BC(St_\chi)\boxtimes
\chi^2\omega_{E/F}$
is an irreducible representation of $\GSO(V).$ Moreover, we have $\theta_\psi(\BC(St_\chi )\boxtimes\chi^2)=0$ and  $$\Theta_\psi(\BC(St_\chi)\boxtimes\chi^2\omega_{E/F})=\theta_\psi(\BC(St_\chi)\boxtimes\chi^2\omega_{E/F})\cong St_\chi|_{\GL_2^+(F)}.$$
\end{enumerate}
\end{thm}
\begin{proof}
See \cite[Section 2.4.3]{lu2016}.
\end{proof}

%% file: super.tex
\subsection{Supercuspidal Representations}
 Assume that $E/F$ is a quadratic field extension.  Set $V_E=E$ to be the quadratic vector space over $F$ and quadratic form coincides with the norm map $N_{E/F}.$ Then the similitude group $\GO(V_E )$ is isomorphic to $E^\times\rtimes \Gal(E/F). $
 
  The following result is well known, referring to   \cite{gan2011endoscopy},~\cite{kudla1996notes}.
 \begin{prop}
 Let $\mu$ be an irreducible representation of $E^\times.$
 If $\mu$ is Galois invariant, then $\mu$ has two extentions
 $\mu^\pm$ to $\GO(V_E),$ in which case,
 only one of them has a nonzero theta lifting to $\GL_2^+$, denoted by $\mu^+.$
 If $\mu$ is not Galois invariant, then $\mu^{+}=ind_{\GSO(V_E)}^{\GO(V_E)}\mu,$
 and $\Theta_\psi(\mu^+) $ is a non-zero irreducible supercuspidal representation of $\GL_2^+(F).$ And
 $ind_{\GL_2^+}^{\GL_2}(\Theta_\psi(\mu^+))$ is irreducible supercuspidal,
 which is dihedral with respect to E.
 \end{prop}
 
 If $\mu=\mu_F\circ N_{E/F}$ for some $\mu_F,$
 then $ind_{\GL_2^+}^{\GL_2}(\Theta(\mu^+))=\pi(\mu_F,\mu_F\omega_{E/F}).$ 
 Moreover, we have $\Theta(\mu^-)=0.$
 In fact, $\mu^+$ is the extension which satisfies $\mu^+|_{\Oo(V_E)}=\mathbf{1}$ and $\mu^-$ is the character of $\GO(V_E)$ satisfying $\mu^-|_{\Oo(V_E)}=\det.$ 
% Consider the dual pair $(GSp_0^+,GO(E)),$
% % ~Hom_{GO(E)}(\Omega',\mu^+)\cong Hom(\mu_F^{-1},\mathbb{C}).$
% there is a $NE^\times\times GO(E)-$equivalent map from
% $S(E)$ to $\mu\otimes\mu_F^{-1}.$ 
\begin{thm}\label{supercus}
Assume that $\tau$ is a supercuspidal representation of $\GL_2(F).$
\begin{enumerate}[(i)]
\item If $\tau$ is not dihedral with respect to $E$, i.e. $\tau|_{\GL_2^+}$ is irreducible, then
$\BC(\tau)$ is a supercuspidal representation of $\GL_2(E)$ and $
\Theta_\psi(\tau|_{\GL_2^+}) =\theta_\psi(\tau|_{\GL_2^+})\cong \BC(\tau)\boxtimes\omega_{E/F}\omega_\tau.$
\item If $\tau\cong \tau\otimes\omega_{E/F}$ is dihedral with respect to $E$ and $\phi_\tau=Ind_{WD_E}^{WD_F}\chi,$ then $\BC(\tau)=\pi(\chi,\chi^\sigma)$ is a principal series of $\GL_2(E),$ where $\sigma$ is the nontrivial element in $\Gal(E/F),$ and  $\chi\neq\chi^\sigma.$ Let $\tau|_{\GL_2^+}\cong\tau^+\oplus\tau^-,$ where $\tau^+$ is $\psi-$generic,
then $\Theta_\psi(\tau^+)=\theta_\psi(\tau^+)\cong \pi(\chi,\chi^\sigma)\boxtimes\omega_{\tau^+}\omega_{E/F}\quad\mbox{and}\quad \theta_\psi(\tau^-)=0.$
%Moreover, $\theta_\psi(BC(\tau)\boxtimes\omega_\tau\omega_{E/F})$ is nonzero.
\item If $\pi=\BC(\tau)$ is a supercuspidal representation of $\GL_2(E),$
then $\pi\boxtimes\omega_\tau\omega_{E/F}$ and $\pi\boxtimes\omega_\tau$ are  supercuspidal representations of $\GSO(V).$ Moreover, we have $$\Theta_\psi(\pi\boxtimes\omega_\tau\omega_{E/F})=\theta_\psi(\pi\boxtimes\omega_\tau\omega_{E/F})=\tau|_{\GL_2^+}~\mbox{ and}~ \theta_\psi(\pi\boxtimes\omega_\tau)=0.$$
\end{enumerate}
\end{thm}
See \cite[Page 307]{gan2011endoscopy} for the Local-Global Argument. 
Then we can use the Howe duality and Proposition \ref{thetaofsup} to deduce Theorem \ref{supercus}.
\begin{rem}
In the case $(ii)$, we may write $$\BC(\tau)\boxtimes\omega_\tau\omega_{E/F}:=\theta^\ast(\tau) \mbox{  
if  } \BC(\tau)\boxtimes\omega_\tau\omega_{E/F}=\theta(\tau^+).$$
If $\tau$ is dihedral, then $\theta^\ast(\tau)=\theta(\tau^+)$ where $\tau^+$ is $\psi-$generic; if $\tau|_{\GL_2^+}$ is irreducible, then
$\theta^\ast(\tau)=\theta(\tau|_{\GL_2^+}).$
\end{rem}

%% file: v-.tex
\subsection{Local Theta Lift from $\GSp^+(W)$ to $\GSO(V^-)$ }\label{sectionV-}
Fix  $\epsilon\in F^\times\setminus NE^\times.$
Now let us consider the $4$-dimensional quadratic space $V^-$ with a quadratic form $q^-,$ where $$~q^-\Big(\begin{pmatrix}
a&x\\\sigma(x)&d
\end{pmatrix}\Big)=\epsilon\cdot (N_{E/F}(x)-ad)~\mbox{for}~v=\begin{pmatrix}
a&x\\\sigma(x)&d
\end{pmatrix}\in V^- .$$ 
Then by Theorem\autoref{similitude}, we have $$\GSO(V^-)=\GSO(1,3)\cong\frac{ \GL_2(E)\times F^\times}{\triangle E^\times},~\mbox{where }~\triangle(t)=(t,N_{E/F}(t^{-1})).$$
%where $E/F$ is a quadratic local field extension,
Similarly, we can  consider the similitude dual pair  $(\GL_2^+,\GSO(1,3)).$ 

We list the results as follow.
\begin{prop}
Assume that $\tau$ is an irreducible smooth representation of $\GSp(W)$ with infinite dimension. 
\begin{enumerate}[(i)]
\item If $\tau$ is not dihedral with respect to E and $\tau\neq \pi(\omega_{E/F},|-|^{\pm1})\otimes\chi$, then $$\Theta_{W,V^-,\psi}(\tau|_{\GL_2^+})\cong\theta_{W,V^-,\psi}(\tau|_{\GL_2^+})\cong \BC(\tau)\boxtimes \omega_{\tau}\omega_{E/F}.$$
%\item if $\tau=\pi(||^{1/2},||^{-1/2}\omega_{E/F}),$ then
%$$\Theta_{W,V^-,\psi}(\tau|_{GL_2^+})\cong \pi(||_E^{-1/2},||_E^{1/2})\boxtimes \mathbb{C},\quad \theta_{W,V^-,\psi}(\tau|_{GL_2^+})\cong St_E\boxtimes \mathbb{C}.$$
\item If $\tau$ equals to $\pi(|-|^{1/2},|-|^{-1/2}\omega_{E/F})\otimes\chi$ or $\pi(|-|^{-1/2},|-|^{1/2}\omega_{E/F})\otimes\chi,$ then
\[\Theta_{W,V^-,\psi}(\tau|_{GL_2^+ })\cong \pi(|-|_E^{1/2},|-|_E^{-1/2})\otimes\chi\circ N_{E/F}\boxtimes\chi\circ N_{E/F}~\mbox{and} ~\theta_{W,V^-,\psi}(\tau|_{GL_2^+})\cong \chi\circ N_{E/F}\boxtimes\chi\circ N_{E/F}. \]
\item If $\tau\cong \pi(\chi,\chi\omega_{E/F}),~\tau|_{GL_2^+}\cong \tau^+\oplus\tau^-,$
where $\tau^+$ is $\psi-$generic and $\tau^-$ is $\psi_\epsilon-$generic. Then
$$\theta_{W,V^-,\psi}(\tau^+)\cong 0\mbox{  and  } \Theta_{W,V^-,\psi}(\tau^-)=\theta_{W,V^-,\psi}(\tau^-)\cong BC(\tau)\boxtimes\omega_{\tau}\omega_{E/F}.$$
%\item if $\tau$ is the Steinberg representation, then $\Theta_{W,V^-,\psi}(St)\cong\theta_{W,V^-,\psi}(St)\cong St_E\boxtimes \omega_{E/F}.$
\item If $\tau\cong\tau\otimes\omega_{E/F}$ is dihedral with respect to $E$
and is a supercuspidal representation of 
$\GL_2(F),$  then $\Theta_{W,V^-,\psi}(\tau^-)=\theta_{W,V^-,\psi}(\tau^-)\cong \BC(\tau)\boxtimes\omega_{\tau}\omega_{E/F}$ and $\theta_{W,V^-,\psi}(\tau^+)=0.$
\end{enumerate}
\end{prop}
%\begin{prop} Assume $\Sigma\cong \pi\boxtimes \chi$ is an irreducible representation of $GSO(1,3),$ then
%$Hom_N(\theta_{V^-,W,\psi}(\Sigma),\psi)\cong Hom_{PD^\times}(\Sigma^\vee,\mathbb{C}).$
%\end{prop}

%% file: prooflocal.tex
\section{Proof of the Main Theorem (Local)}\label{seclocalproof}
In this section, we give the proof of the \textbf{Main Theorem (Local)}. %\autoref{localmain}\autoref*{localmain}
 The key idea is to transfer the period problem $\Hom_{D^\times( F)}(\pi,\mathbb{C}),$ where $\pi$ is an irreducible smooth representation of $D^\times(E),$  to the period problem $\Hom_{PD^\times(F)}(\Sigma,\mathbb{C}),$ where $\Sigma$ 
 is a representation of $\GSO(V)$ associated to $\pi.$ Due to Proposition \ref{whittaker}, the latter one is related to the nonvanishing Whittaker model of the big theta lift $\Theta_\psi(\Sigma).$
 
 %\paragraph*{}
Recall that for $a\in F^\times,$ and $y_a=\begin{pmatrix}
a&0\\0&-1
\end{pmatrix}\in V,$ Proposition \ref{whittaker} tells us that
\[\Hom_{\SO(y_a^\perp)}(\Sigma,\mathbb{C})\neq0~\mbox{if and only if}~\Hom_N(\Theta_\psi(\Sigma),\psi_a)\neq0. \]
\begin{proof}
[\textbf{Proof of Main Theorem (Local)}]
Assuming that $\omega_\pi|_{F^\times}=1,$ then there exists a character $\mu$ of $E^\times$ such that
$\omega_\pi=\mu^\sigma/\mu$ by Hilbert's Theorem 90.
\begin{enumerate}[(1)] 
	\item We first prove that $(i)$ implies $(ii).$ Assume that $\pi$ is $\GL_2(F)$-distinguished.
	 Pick $a=1,$ then $y_a$ corresponds to the split quaternion algebra by Lemma \ref{DE}. Due to Theorem \ref{SO(3)},
	we have $\SO(y_a^\perp,F)\cong \PGL_2(F).$
	Set $\Sigma=\pi\otimes\mu\boxtimes\mu|_{F^\times},$ then $\omega_\pi\cdot\mu^2=\mu\circ N_{E/F}$ and
	\[\Hom_{\PGL_2(F)}(\Sigma,\mathbb{C})=\Hom_{\PGL_2(F)}(\pi\otimes\mu\boxtimes\mu|_{F^\times},\mathbb{C})=\Hom_{\GL_2(F)}(\pi\cdot\mu\circ\det,\mu|_{F^\times}\circ\det)\quad (\dagger)\]
	i.e. $\Hom_{\SO(y_a^\perp,F)}(\Sigma,\mathbb{C})=\Hom_{\GL_2(F)}(\pi,\mathbb{C} )$ is nonzero. By Proposition \ref{whittaker},  the representation
	$\Theta_\psi(\Sigma)$ of $\GSp^+(W)$ is $\psi$-generic. By what we have shown in Section \ref{sec3}, %the representation  %$\pi\otimes\mu$ of $GL_2(E)$ must be of form $BC(\tau)$ for some representation $\tau$ of $GL_2(F),$ and 
we get	$$\Sigma=\BC(\tau)\boxtimes\omega_{E/F}\omega_{\tau},$$ i.e. $\pi\otimes\mu=\BC(\tau)$ for some representation $\tau$ of $\GL_2(F)$ and $\omega_\tau=\mu|_{F^\times}\omega_{E/F}.$
	\par
	Conversely, if $\pi=\BC(\tau)\otimes\mu^{-1},$ 
	then $\Sigma=\BC(\tau)\boxtimes\mu|_{F^\times}$
	 and the theta lift $\Theta_\psi(\Sigma)$ is $\psi$-generic.
	By Proposition \ref{whittaker},
	one can see that $\Sigma$ is $\SO(y_a^\perp)$-distinguished, i.e $\Hom_{\PGL_2(F)}(\Sigma,\mathbb{C})\neq0.$ By the identity $(\dagger),$ we obtain that $\pi$ is $\GL_2(F)$-distinguished.
	\par
	If $\pi=\BC(\tau)\otimes\mu^{-1},$ we set $\phi_\tau:WD_F\rightarrow \GSp_2(\mathbb{C})=\GL_2(\mathbb{C})$ to be the Langlands parameter of $\tau.$ Assume that $B_\tau$
	is the non-degenerate symplectic bilinear form. Then the Langlands parameter $\phi_\pi$ with respect to $\pi$ is equal to $\phi_\tau|_{WD_E}\cdot\mu^{-1}$ up to conjugacy and $\mu\mu^\sigma=\det\phi_\tau|_{WD_E}.$ Assume that $s\in W_F\setminus W_E.$  Set $B_\pi(m,n)=B_\tau(m,\phi_\tau(s^{-1})n).$ 
	It is easy to check that $B_\pi$ is conjugate-orthogonal.
	\par
	Conversely, we want to show $(iii)$ implies $(ii).$
	Since $(\phi_\pi\otimes\mu)^\sigma=\phi_\pi^\sigma\mu^\sigma=\phi_\pi^\vee\mu^\sigma=\phi\mu,$
	 there is a lift $\phi_\tau:WD_F\rightarrow \GL_2(\mathbb{C})$
	such that $\phi_\tau|_{WD_E}=\phi_\pi\otimes\mu,$ i.e. $\BC(\tau)=\pi\otimes\mu.$
	 Moreover, we can get $\det\phi_\tau(s)=-\mu(s^2) $ since   $\phi_\pi$ is conjugate-orthogonal.
	\item We use the same method. Comparing with the theta lift from $\GSO(V)$ to $\GSp^+(W),$ we know $\Theta_\psi(\Sigma)$ must be $\psi$-generic if it is nonzero. The only trouble case is that $\tau$ is dihedral.
	Pick $a=\epsilon\in F^\times\setminus NE^\times,$
	then $y_a$ corresponds to the non-split quaternion algebra, and $\SO(y_a^\perp,F)\cong PD^\times(F).$
	Then $\pi$ is $D^\times(F)$-distinguished if and only if $\Theta_\psi(\Sigma)$ is $\psi_\epsilon$-generic for $\Sigma=\pi\otimes\mu\boxtimes\mu|_{F^\times}.$ 
%	If $\pi$ is a discrete series of $GL_2(E),$ then $\tau|_{GL_2^+}$ is irreducible. If $\pi$ is a principal series,
\par
Now, we prove that $(i)$ implies $(ii).$ Assume that $\pi$ is $D^\times(F)$-distinguished.
 If $\Theta_\psi(\Sigma)$ is $\psi_\epsilon$-generic, then $\Theta_\psi(\Sigma)\neq0$ and there exists a representation $\tau$
 of $\GL_2(F)$ such that $\Theta_\psi(\Sigma)=\tau|_{\GL_2^+}$ is irreducible. Moreover, by the proof of (1),
 we can get $\pi=\BC(\tau)\otimes\mu^{-1}$ which is $\GL_2(F)$-distinguished. The condition that $\tau|_{\GL_2^+}$ is irreducible means that $\tau$
 can not be dihedral supercuspidal  with respect to $E$ by Proposition \ref{orthV}(iii),
  i.e. $\pi$ can not be of the form $\pi(\chi_1,\chi_2)~\mbox{ where}~
 \chi_1\neq\chi_2 ~\mbox{and}~ \chi_1|_{F^\times}=\chi_2|_{F^\times }=1.$
 \par
 Conversely, if $\pi$ is $\GL_2(F)$-distinguished and not in the case of Proposition \ref{orthV}(iii),
 then we can find a representation $\tau$ of $\GL_2(F)$ such that $\tau|_{\GL_2^+}$ is irreducible and participates in the theta correspondence with $\GSO(V).$ Hence $\tau|_{\GL_2^+}$ is $\psi_\epsilon$-generic, which  implies that
 $\pi$ is $D^\times(F)$-distinguished.
 \item First, 
 $(ii)$ and $(iii)$ are equivalent
 since the group $PD^\times(F)$ is compact. Second,
 we prove that $(iii)$
 implies $(i).$ Assume that $\pi$ is $D^\times(F)$-distinguished and supercuspidal, then $\pi|_{\GL_2(F)}$ is a projective object in the category of smooth representations of $\PGL_2(F),$ then $\Ext^1_{\PGL_2(F) }(\pi|_{\GL_2(F)},\mathbb{C} )=0$
 and $\dim \Hom_{\PGL_2(F)}(\pi|_{\GL_2(F)},\mathbb{C} )=1.$ If  $\pi$ is the twisted Steinberg representation and
 is $D^\times(F)$-distinguished, then  $\pi=St_E\otimes\chi,$ where $\chi|_{F^\times}=\omega_{E/F}.$ Since $\Ext^1_{\PGL_2(F)}(\chi|_{F^\times},\mathbb{C} )=0,$  the Mackey theory implies that 
 \[\Ext^1_{\PGL_2(F) }(\pi|_{\GL_2(F)},\mathbb{C} )\cong \Ext_{E^\times/F^\times}^1(\chi\circ N_{E/F},\mathbb{C} )=0. \]
 If $\pi=\pi(\chi_1,\chi_2)$ is an irreducible principal series and is $D^\times(F)$-distinguished, then $\chi_1\chi_2^\sigma=1$ and by the Mackey theory, we have
 \[\dim \Hom_{\PGL_2(F)}(\pi|_{\GL_2(F)} ,\mathbb{C})=1+\dim \Ext^1_{\PGL_2(F)}(\pi|_{\GL_2(F)},\mathbb{C}). \]
% then the conclusion holds.
 \par
 Finally, if $\pi$ is $\GL_2(F)$-distinguished and is not $D^\times(F)$-distinguished,
 then $\pi=\pi(\chi_1,\chi_2)$ is a principal series associated with two distinct characters $\chi_1,\chi_2$ of $E^\times,$ and $\chi_1|_{F^\times}=\chi_2|_{F^\times}=1.$
 By a similar reason, using the Mackey theory, we have $$\dim \Hom_{\PGL_2(F)}(\pi|_{\GL_2(F)},\mathbb{C} )=\dim \Ext^1_{\PGL_2(F)}(\pi|_{\GL_2(F)},\mathbb{C} ) .$$
 %If $\pi$ is an irreducible principal series, then by the Mackey theory and Frobenius Reciprocity, we are done.
 %Simiarly, we can show that the conditions $(iii)$ and $(i)$ are equivalent.
\end{enumerate}
If a character $\mu_1:E^\times\rightarrow\mathbb{C}^\times$ also satisfies $\omega_\pi=\mu_1^\sigma/\mu_1$ and $\pi=BC(\tau)\otimes\mu^{-1},$ we have
\[\mu_1\mu^\sigma=\mu_1^\sigma\mu, \]
which means $\mu_1\mu^\sigma$ is Galois invariant, hence it factors through the norm map. Assume that
$$\mu_1\mu^\sigma=\mu_F\circ N.$$ 
Then setting $\tau_1=\tau\otimes\mu_F\mu^{-1}|_{F^\times},$ we have
$\omega_{\tau_1}=\omega_{E/F}\mu_1|_{F^\times}$ and 
\[\pi\otimes\mu_1=\pi\otimes\frac{\mu_F\circ N}{\mu^\sigma}=BC(\tau)\otimes\frac{\mu_F\circ N}{\mu\mu^\sigma}= BC(\tau_1). \]
%and $\omega_$
Hence, this theorem does not depend on the choice of the character $\mu.$ 
It finishes the proof.
\end{proof}
\begin{rem}
	If we consider the Euler-Poincare pairing \cite{prasad2013ext} $$EP_{PD^\times(F)}(\pi,\mathbb{C})=\dim \Hom_{PD^\times(F)}(\pi,\mathbb{C} )-\dim \Ext_{PD^\times(F)}^1(\pi,\mathbb{C} ),$$
	then given an irreducible smooth representation $\pi$ of $\GL_2(E)$
	with $\omega_\pi|_{F^\times}=1,$ we have
	\[EP_{\PGL_2(F)}(\pi|_{\GL_2(F)},\mathbb{C})=1\mbox{  if  and  only  if  }EP_{PD^\times(F)}(\pi|_{D^\times(F)},\mathbb{C} )=1. \]
\end{rem}
\begin{rem}
	For the archimedean case $F=\mathbb{R}$ and $E=\mathbb{C}$, the results for period problems  also hold.
	Fix a nontrivial additive character $\psi,$ Cognet \cite{cognet1986representation} proved that the small theta lift of infinite dimensional representations from $\GL_2^+(\mathbb{R})$ (matrix with positive determinant) to the similitude orthogonal group $$\GSO(3,1)\cong \frac{\GL_2(\mathbb{C})\times\mathbb{R}^\times }{\triangle\mathbb{C}^\times}$$ has the same pattern with the nonarchimedean situations.
R. Gomez  and C-B Zhu \cite{zhu2015} showed that for the unipotent subgroup $N$ of $\GL_2^+(\mathbb{R}),$ the isomorphisms
	\[\Hom_{N}(\Theta_\psi(\Sigma),\psi)\cong \Hom_{\PGL_2(\mathbb{R} )}(\Sigma^\vee,\mathbb{C})~ \mbox{and}~ \Hom_{N}(\Theta_\psi(\Sigma),\psi_\epsilon)\cong \Hom_{P\mathbb{H}^\times}(\Sigma^\vee,\mathbb{C}) \] 
	hold for the Casselman-Wallach representation $\Sigma$ of $\GSO(3,1),$ where $\epsilon=-1$ and $\mathbb{H}$ is the real Hamilton algebra. Then the proof to the local period for the archimedean case is almost the same. %In fact, they has a more general result for the Whittaker models associated to nilpotent orbits.
\end{rem}
\begin{coro}\cite[Page 162]{flicker1991distinguished} Let $E$ be a quadratic extension of a nonarchimedean local field $F.$
	If $\pi=\pi(\chi_1,\chi_2)$ is an irreducible principal series of $\GL_2(E),$ then $\pi$ is $\GL_2(F)$-distinguished if and only if one of the following holds:
	\begin{itemize}
		\item $\chi_1\chi_2^\sigma=1$ or
		\item $\chi_1$ and $\chi_2$ are two distinct characters of $E^\times$ with $\chi_1|_{F^\times}=\chi_2|_{F^\times}=1.$
	\end{itemize}
\end{coro}
\begin{proof}
	If $\pi$ is $\GL_2(F)$-distinguished, by the Main Theorem (Local), there exists a character $\mu$ and a representation $\tau$ such that $\pi=\BC(\tau)\otimes\mu^{-1}$ with $\omega_\pi=\mu^\sigma/\mu$ and $\omega_{\tau}=\omega_{E/F}\mu|_{F^\times}.$ Then $\pi(\chi_1\mu,\chi_2\mu)$ is Galois invariant and
	$\mu^\sigma=\mu\chi_1\chi_2.$ Then either $(\chi_1\mu)^\sigma=\chi_2\mu$ and $\chi_1\neq\chi_2$ or  $\chi_1\mu$ and $\chi_2\mu$ both factor through the norm map. For the first case, we have $$\chi_1|_{F^\times}=\chi_2|_{F^\times}~\mbox{ and}~~ \mu|_{F^\times}=\omega_\tau\cdot\omega_{E/F}=(\chi_1\mu)|_{F^\times}\omega_{E/F}\cdot\omega_{E/F},~\mbox{i.e.}~\chi_1|_{F^\times}=1.$$
	For the second case, we have $(\chi_1\mu)^\sigma=\chi_1\mu,$ then $\chi_1^\sigma\chi_2=1.$
	
	Conversely, if $\chi_1$ and $\chi_2$  are trivial on $F^\times,$ set $\pi=\pi(\frac{\mu_1^\sigma}{\mu_1},\frac{\mu_2^\sigma}{\mu_2})$ and $\mu=\mu_1\mu_2.$ Since $\chi_1\neq\chi_2,$  the character $\mu_1\mu_2^\sigma$ of $E^\times$ does not factor through the norm. Let $\tau$ to be the dihedral supercuspidal representation of $\GL_2(F)$ with respect to $(E,\mu_1\mu_2^\sigma).$ Then $\omega_\tau=\mu|_{F^\times}\omega_{E/F}$ and $\pi=\BC(\tau)\otimes\mu^{-1}$ is $\GL_2(F)$-distinguished.
	If $\chi_1\chi_2^\sigma=1,$ set $\mu=\chi_2^\sigma$ and $\tau=\pi(\omega_{E/F},\chi_2|_{F^\times} ),$
	then $\pi=\BC(\tau)\otimes\mu^{-1}$ is $\GL_2(F)$-distinguished. 
\end{proof}
\begin{coro}
	For the twisted Steinberg representation $\pi=St_E\otimes\chi$ of $\GL_2(E),$ the following statements are equivalent:
	\begin{enumerate}[(i)]
		\item $\pi$ is $\GL_2(F)$-distinguished;
		\item $\pi$ is $D^\times(F)$-distinguished;
		\item the Langlands parameter $\phi_\pi$ is conjugate-orthogonal;
		\item $\chi|_{F^\times}=\omega_{E/F}.$
	\end{enumerate}
\end{coro}
More generally, we can consider the conditions for $(\GL_2(F),\chi)$-period problems, i.e.
$\Hom_{\GL_2(F) }(\pi,\chi ). $
\par
Assume that $\omega_\pi|_{F^\times}=\chi^2$. One may pick a character $\chi_E$ of $E^\times$ such that $\chi_E|_{F^\times}=\chi^{-1}$. Then
\[\dim \Hom_{\GL_2(F)}(\pi,\chi )=\dim \Hom_{\GL_2(F)}(\pi\otimes\chi_E,\mathbb{C} ) \]
which has been discussed before.

%% file: global.tex
\section{Global Theta Lift for Similitude Groups}
Let $F$ be a number field. %Assume $E$  is a quadratic extension of $F$. 
Assume that $E$ is a quadratic field extension of $F.$  %$\mathbb{A}_E=\mathbb{A}\otimes E$
%is the adele ring of $E.$
%Let W be a symplectic vector space and let V be an orthogonal vector space.
Let $\mathbb{A}$ be the adele ring of $F$
and $\mathbb{A}_E=\mathbb{A}\otimes_F E.$  Let us fix %$\psi:F\backslash\mathbb{A}\rightarrow S^1$ 
 a unitary additive character $\psi:F\backslash\mathbb{A}\rightarrow\mathbb{C}^\times.$
 Let $W$ be a $2$-dimensional symplectic vector space over $F$ and let $V=F\oplus E\oplus F$ be a $4$-dimensional quadratic space over $F$ with quadratic form $q(e,a,b)=N_{E/F}(e)-ab.$
Let $\omega_{\psi}$ be the Weil representation for the dual pair $\Sp(W,\mathbb{A})\times \SO(V,\mathbb{A}).$

Recall \[R=\GSp^+(W)\times \GSO(V)~\mbox{and }~R_0=\{(g,h)\in R~| \lambda_W(g)\cdot \lambda_V(h)=1 \} \]
and the actions are the same as in the local setting, see Section \ref{sec3}.

For a Schwartz function $\phi\in S(V,\mathbb{A})$ and $(g,h)\in R_0(\mathbb{A}),$ set
\[ \theta_\psi(\phi)(g,h)=\sum_{x\in  V(F)}\omega_\psi(g,h)\phi(x).
\] 
Then $\theta_\psi(\phi)$ is a function of moderate growth on $R_0(F)\backslash R_0(\mathbb{A}).$
Suppose that $\pi\boxtimes\mu$ is a cuspidal automorphic representation of $\GSO(V,\mathbb{A})$ and $f\in \pi\boxtimes\mu,$
we set
\[\theta_\psi(\phi,f)(g)=\int_{\SO(V,F)\backslash \SO(V,\mathbb{A})}\theta_\psi(\phi)(g,h_1h)\cdot\overline{f(h_1h)}dh
\]where $h_1$ is any element of $\GSO(V,\mathbb{A})$ such that $sim_V(h_1)=sim_W(g^{-1}).$ Then
$$\Theta_\psi(\pi\boxtimes\mu)=\Big<\theta_\psi(\phi,f):\phi\in S(V,\mathbb{A}),~f\in\pi\boxtimes\mu\Big>$$
is an automorphic representation (possibly zero) of $\GSp(W)^+(\mathbb{A}).$ Considering the Fourier coefficient
of $\theta_\psi(\phi,f)$ with respect to $\psi_a,$ where $\psi_a(x)=\psi(ax),$ and $a$ is an arbitrary nonzero element in $\mathbb{A},$ we have
\begin{equation*}
\begin{split}
&Wh_{N,\psi_a}(\theta_\psi(\phi,f))\\
=&\int_{N(F)\backslash N(\mathbb{A})}\overline{\psi_a(u)}\int_{\SO(V,F)\backslash \SO(V,\mathbb{A})}\theta_\psi(\phi)(u,h)\cdot\overline{f(h)}dhdu\\
=&\int_{\SO(V,F)\backslash \SO(V,\mathbb{A})}\overline{f(h)}\cdot\int_{N(F)\backslash N(\mathbb{A})}\overline{\psi(au)}\cdot\sum_{x\in V(F)}\omega_\psi(u,h)\phi(x)dudh\\
=&\int_{\SO(V,F)\backslash \SO(V,\mathbb{A})}\overline{f(h)}\cdot\sum_{x\in V^a}\phi(h^{-1}x)dh
\end{split}
\end{equation*}
where $V^a=\{x\in V(F)|~q(x)=a \}=\SO(V,F)\cdot\{y_a\}$ and the stabilizer of $y_a$ is $\SO(y_a^\perp,F).$
Hence, we have\begin{equation*}
\begin{split}Wh_{N,\psi_a}(\theta_\psi(\phi,f))=&\int_{\SO(V,F)\backslash \SO(V,\mathbb{A})}\overline{f(h)}\cdot\sum_{\gamma\in ~\SO(y_a^\perp,F)\backslash \SO(V,F)}\phi{(h^{-1}\gamma^{-1}y_a )}dh\\
=&\int_{\SO(y_a^\perp,F)\backslash \SO(V,\mathbb{A})}\overline{f(h)}\phi(h^{-1}y_a)dh\\
=&\int_{\SO(y_a^\perp,\mathbb{A})\backslash \SO(V,\mathbb{A}) }\phi(h^{-1}y_a)\cdot\int_{\SO(y_a^\perp,F)\backslash \SO(y_a^\perp,\mathbb{A}) }\overline{f(th)}dtdh.
\end{split}\end{equation*}
That means $\theta_\psi(\phi,f)$ has non-zero $(N,\psi_a)$-period
if and only if $$\int_{\SO(y_a^\perp,F)\backslash\SO(V,\mathbb{A})}\overline{f(h)}\phi(h^{-1}y_a)dh\neq0.$$
%If $\pi=ind_{B(\mathbb{A}_E)}^{GL_2(\mathbb{A}_E)}\chi,$
%where $\chi:{E}^\times\backslash\mathbb{A}_E^\times\rightarrow\mathbb{C}^\times$ is a unitary character,

We define the period integral $$P_{\SO(y_a^\perp)}(f):=\int_{\SO(y_a^\perp,F)\backslash \SO(y_a^\perp,\mathbb{A})}f(t)dt,$$ from \cite[Proposition 5.2]{ganshimura}, we know that
$Wh_{N,\psi_a}(\theta_\psi(\phi,f))$  is nonzero for some $f$ and $\phi$ if and only if the period integrate
$P_{\SO(y_a^\perp)}$ is nonzero on $\pi.$ 
%We repeat the proof as follow.
\begin{prop}\cite[Proposition 5.2]{ganshimura} The $\psi_a$-coefficient of $\Theta_{\psi}(\pi)$ is nonzero  if and only if the period integral $P_{\SO(y_a^\perp)}$ is nonzero on $\pi.$ In particular, if $P_{\SO(y_a^\perp)}$ is nonzero on $\pi$ for some
	$a,$ then the global theta lift $\Theta_\psi(\pi)$ is nonzero.\label{periodsto}
\end{prop}
This is a global analogue of Proposition \ref{whittaker}.

\begin{coro}
If $\theta_\psi(\phi,f)=0,$ then $P_{\SO(y_a^\perp)}$ vanishes on $\pi.$
\end{coro}
\subsection*{Automorphic realization of Mixed Model}
Recall,  we have a Witt decomposition
\[W=X+ Y~\mbox{and}~V=Fv_0+ V_E+ Fv_0^\ast \]
where $V_E=E$ is the quadratic space with a quadratic form $e\mapsto N_{E/F }(e).$ The stabilizer
$Q=TU$ of $v_0$ in $\SO(V)$ is a Borel subgroup of $\SO(V),$
with \[T=\GL(Fv_0)\times E^1~\mbox{and}~U\subset \Hom(v^\ast, V_E). \]
We can consider the maximal isotropic subspace
\[\mathbb{X}= W\otimes v_0^\ast+ Y\otimes V_E\subset W\otimes V \]
and write an element of $\mathbb{X}$ as
\[(w,e)=(x,y,e)\in W+V_E=X+Y+V_E. \]
Assuming that $\omega_\psi$ is the Weil representation of $\Sp(W,\mathbb{A})\times \SO(V,\mathbb{A}),$ then
extend it to  $R_0(\mathbb{A}).$ 
And the action of $\GSp^+(W,\mathbb{A})\times \GSO(V,\mathbb{A})$ is defined as before, See Section \ref{sec3}.

Given a cuspidal representation $\tau$ of $\GSp^+(W),$ set $f\in\tau.$ 
For $\phi'\in S(X\otimes V,\mathbb{A}),$ assume that the partial Fourier transform function is $\phi=I(\phi')\in S(\mathbb{X},\mathbb{A}).$
Define
\[\theta_\psi(\phi,f)=\int_{\Sp(W,F)\backslash \Sp(W,\mathbb{A} ) }\overline{f(g_1g)}\sum_{x\in\mathbb{X}(F) }\omega_\psi(g_1g,h)\phi(x)dg , \]
where $\lambda_W(g_1)=\lambda_V(h).$
Then
$\theta_\psi(\phi,f)$ is an automorphic representation of $\GSO(V)(\mathbb{A}),$
the unipotent radical $U(\mathbb{A})\cong \mathbb{A}_E.$
Consider the subset $\{(x,0,x^{-1})\}\subset (W-\{0\})+ V_E,$ then
\begin{equation*}
\begin{split}
&Wh_{U,\psi_E}(\theta_\psi(\phi,f))\\
=&\int_{U(F)\backslash U(\mathbb{A})}\overline{\psi_E(u)}\int_{\Sp(W,F)\backslash \Sp(W,\mathbb{A})}\theta_\psi(\phi)(u,g)\cdot\overline{f(g)}dgdu\\
=&\int_{U(F)\backslash U(\mathbb{A})}\overline{\psi_E(u)}\cdot \int_{\Sp(W,F)\backslash \Sp(W,\mathbb{A})}\overline{f(g)}\cdot\sum_{x\in W+V_E}\omega_\psi(u,g)\phi(x)dgdu\\
%=&\int_{U(F)\backslash U(\mathbb{A})}\overline{\psi_E(u)}\cdot\int_{Sp(W,F)\backslash Sp(W,\mathbb{A})}\overline{f(g)}\cdot\\
%%&\quad \sum_{(x,y,v)\in W+V_E}\psi(-xy N(u)+y\cdot tr(u\sigma(v)))\cdot(\omega_0(g)\phi(g^{-1}\begin{pmatrix}
%x\\y
%\end{pmatrix}))(v-uy)dgdu\\
%=&\int_{U(F)\backslash U(\mathbb{A})}\overline{\psi_E(u)}\cdot\int_{Sp(W,F)\backslash Sp(W,\mathbb{A})}\overline{f(g)}\cdot\sum_{\gamma\in N(F)\backslash Sp(W,F)} \\
%&\quad\sum_{v\in V_E}\psi(-xy N(u))\psi_E(y\cdot u\sigma(v)))(\omega_0(g)\phi(g^{-1}\begin{pmatrix}
%1\\0
%\end{pmatrix}\gamma ))(v)dgdu\\
%=&\int_{U(F)\backslash U(\mathbb{A})}\overline{\psi_E(u)}\cdot\int_{N(F)\backslash Sp(W,\mathbb{A})}\overline{f(g)}\cdot \sum_{v\in V_E}\psi_E( u\sigma(v))(\omega_0(g)\phi(g^{-1}\begin{pmatrix}
%1\\0
%\end{pmatrix}))(v)dgdu\\
=&\int_{N(F)\backslash \Sp(W,\mathbb{A})}\overline{f(g)}(\omega_0(g)\phi(g^{-1}\begin{pmatrix}
1\\0
\end{pmatrix}))(1)dg\\
=& \int_{N(\mathbb{A})\backslash \Sp(W,\mathbb{A})}\omega_\psi(g)\phi(1,0,1)\cdot\int_{N(F)\backslash N(\mathbb{A})}\overline{f(n(x)g)}\psi(x)dxdg.
\end{split}
\end{equation*}
\begin{coro}
If $\tau$ is $\psi$-generic, then $\theta_\psi(\phi,f)$ is $\psi_E$-generic and hence nonzero. Moreover, if $\theta_{V,W,\psi}(\tau)$ is cuspidal and $\tau$ is not $\psi$-generic, then $\theta_{V,W,\psi}(\tau)=0.$
\end{coro}
The proof is very similar with the proof in \cite[Proposition 5.2]{ganshimura}.

%% file: global2.tex
Assume that $\Sigma=\pi\boxtimes\chi$ is an irreducible cuspidal representation of $\GSO(V,\mathbb{A}),$ %fix an additive character
%$\psi:F\backslash \mathbb{A}\rightarrow \mathbb{C}^\times.$ 
where $\pi$ is a cuspidal automorphic representation of $\GL_2(\mathbb{A}_E).$
Let S be a finite set containing archimedean places of $F,$ 
such that $\Sigma_v$ is spherical and $\psi_v$ is unramified for $v\notin S.$ 

%Let $a=1,$
\begin{prop}\label{nonvanishing}
The following statements are equivalent:
\begin{enumerate}[(i)]
\item the theta lift $\theta_\psi(\Sigma)$ from $\GSO(V,\mathbb{A})$ of $\GSp^+(W,\mathbb{A})$ is nonzero;
\item  $\Sigma=\BC(\tau)\boxtimes \omega_\tau\omega_{E/F}$ for some automorphic cuspidal representation $\tau$ of $\GL_2(\mathbb{A});$
\item the partial L-function $L^S(s,\Sigma,Std)$ has a pole at $s=1.$ 
\end{enumerate}
Moreover, in which case, the cuspidal representation $\tau^+=\theta_\psi(\Sigma)$ of $\GL_2^+(\mathbb{A})$ is $\psi$-generic and then $\Sigma$ is $\PGL_2(\mathbb{A})$-distinguished.
\end{prop}
\begin{proof} First, we note that $(iii)\Rightarrow(i)$ comes from the same argument as \cite[Corollary 7.9]{gan2011regularized}.
\par
 Now, we want to show $(i)$ implies $(ii).$ If $\Sigma$ participates in the theta correspondence with $\GSp^+(W,\mathbb{A})$, then $\theta_{\psi_v}(\Sigma_v)\neq0$ and
$\Sigma_v\cong \BC(\tau_v)\boxtimes\omega_{\tau_v}\omega_{E_v/F_v}$ at each place $v$ of $F.$ The theta lift $\theta_{\psi}(\Sigma)$ is cuspidal since its constant term is zero.
Assume that $\tau\subset Ind_{\GL_2^+(\mathbb{A})}^{\GL_2(\mathbb{A})}\Theta_\psi(\Sigma)$ is the cusp constituent of $\GL_2(\mathbb{A}),$ which is globally $\psi$-generic. And $\theta^\ast(\tau)$ from $\GL_2(\mathbb{A})$ to $\GSO(3,1)(\mathbb{A})$ is a cuspidal automorphic representation of $\GSO(3,1)(\mathbb{A})$ due to Rallis tower property. By the multiplicity one theorem for $\GL_2(\mathbb{A}_E)$,  we have $$\Sigma\cong\theta^\ast(\tau)\cong \BC(\tau)\boxtimes\omega_\tau\omega_{E/F}.$$
\par 
For $(ii)\Rightarrow(iii),$ if $\Sigma=\BC(\tau)\boxtimes\omega_\tau\omega_{E/F},$ then via  local computations,
we have $$L^S(s, \BC(\tau)\boxtimes\omega_\tau \omega_{E/F}, Std)=L^S(s, \tau, Ad\otimes\omega_{E/F})\zeta^S_F(s).$$ By the result of Shahidi \cite[Theorem 5.1]{shahidi1981certain}, $L^S(s,\tau,Ad\otimes\omega_{E/F})$
is nonzero at $s=1.$ Then $L^S(s,\Sigma,Std)$ has a pole at $s=1.$
\par
From the proof above, we know $\tau^+=\theta_\psi(\Sigma)=\theta_\psi(\BC(\tau)\boxtimes\omega_\tau\omega_{E/F} )$ is $\psi$-generic. By \cite[Proposition 5.2]{ganshimura}, we pick $a=1$ and $y_a=\begin{pmatrix}
1&0\\0&-1
\end{pmatrix}\in V,$ then $\SO(y_a^\perp,\mathbb{A})\cong \PGL_2(\mathbb{A} )$ and $\Sigma$
is $\PGL_2(\mathbb{A})$-distinguished.%If $\Sigma= BC(\tau)\boxtimes$
\end{proof}

 \section{Proof of the Main Theorem (Global)}
 In this section, we give the proof of  \textbf{Main Theorem (Global)}.
 
 Let us recall the notation. Let $E$ be a quadratic  extension of a number field $F.$
Let $W$ be a $2$-dimensional symplectic space over $F.$
Let $D$ be a quaternion algebra over $F.$
Let $Z(\mathbb{A})=\mathbb{A}^\times$
be the center of $D^\times(\mathbb{A} ).$ %defined over a number field $F.$
Let $\pi^D$
be a cuspidal automorphic irreducible representation of $D^\times(\mathbb{A}_E),$ with central character $\omega_{\pi^D}$ and $\omega_{\pi^D}|_{\mathbb{A}^\times}=1.$ 
%$Res_{E/F}(D')(\mathbb{A}),$
%where $(Res_{E/F}D')(F)\cong D(F)\otimes_F E.$ 
Assume that the Jacquet-Langlands
correspondence representation $\pi=JL(\pi^D)$ of $\pi^D$  is a cuspidal automorphic representation
of $\GL_2(\mathbb{A}_E).$ 
Let $S$ to be a finite set of places containing archimedean places in $F,$ such that for all $ v\notin S,$   $\pi_v^D$ is unramified and $E_v$ is unramified. Since $D\otimes_F E$ may not be isomorphic to $M_2(E),$ we can not find a point in $V=V_E\oplus\mathbb{H}$ corresponding to $D$ by Lemma \ref{DE}. Then we need another $4$-dimensional quadratic space over $F.$ 
\par
Set $D=\Big(\frac{c,d}{F} \Big)$ and $E=F[\sqrt{d}].$ %$V=V_E+\mathbb{H},$ %
 Assume that $X_D=~\{x\in D\otimes_F E| x^\ast=\sigma(x) \}$ with a quadratic form
 \[q(v)=x_1^2-d(cx_2^2+bx_3^2-bcx_4^2),v=(x_1,x_2,x_3,x_4)\in X\cong F^4. \]
We consider the theta lifting $\theta_{X_D,W,\psi}(\phi,f)$ from
$\GSO(X_D,\mathbb{A})$ to $\GSp^+(W,\mathbb{A}).$
Pick $y=1\otimes1\in X_D,$ then $\SO(y^\perp,F )\cong PD^\times(F)$ by Theorem \ref{SO(3)}.
Let $\Sigma^D=\pi^D\otimes\mu\boxtimes\mu|_{\mathbb{A}^\times}$ be the cuspidal automorphic representation of $\GSO(X_D,\mathbb{A}).$ Respectively, we set $\Sigma=\pi\otimes\mu\boxtimes\mu|_{F^\times}$ to be the cuspidal representation of $\GSO(V,\mathbb{A}).$
For $(g,t)\in D^\times(\mathbb{A}_E)\times\mathbb{A}^\times  $ and $f\in\Sigma^D,$
we define $$f(g,t)=f^D(g)\mu(N(g))\cdot\mu(t),~\mbox{where}~f^D\in\pi^D.$$
Then we have $f(d,N(d)^{-1})=f^D(d)$ and
\[\int_{Z(\mathbb{A})D^\times(F)\backslash D^\times(\mathbb{A}) }f^D(t)dt=\int_{PD^\times(F)\backslash PD^\times(\mathbb{A}) }f(h)dh. \]
From \cite[Proposition 5.2]{ganshimura}, one can see that $\Sigma^D$ is $PD^\times(\mathbb{A})$-distinguished if and only if $\theta_{X_D,W,\psi}(\Sigma^D)$ is $\psi$-generic.
In fact, the only difference between $\Sigma^D$ and $\Sigma$ is at the local places $v$ where $D_v$ ramified and $E_v$ split. In this case, $\Sigma_v$ is $PD^\times_v(F_v)$-distinguished if and only if $\Sigma_v=\tau\boxtimes\tau^\vee$ for some representation $\tau$ of $D^\times(F_v).$ So the key point is the local results at the places $v$ where $D_v$ ramified and $E_v$ non-split, $X_{D,v}\cong V^-,$ which have been classified in Section \ref{seclocalproof}.
%And we use $L(s,\Sigma^D )$ to denote the standard $L$-function $L(s,\Sigma^D,Std),$ and similarly for the other $L$-functions.
 %then the following statements are equivalent
%\begin{enumerate}[(i)]
%\item $\pi^D$ is $D^\times(\mathbb{A})$-distinguished;
%\item $\Sigma^D=\pi^D\otimes\mu\boxtimes\mu|_{\mathbb{A}^\times}$ is $PD^\times(\mathbb{A} )$-distinguished for some $\mu:\mathbb{A}_E^\times\rightarrow\mathbb{C}^\times$;
%\item $L(1,\Sigma^D,Std)=\infty$ , and %$Wh_{N_v,\psi_v}(\theta_{\psi_v}(\Sigma^D_v))\neq0$ 
%for  each place $v$ of $F$ where $D_v$ is ramified, and $E_v$
%is not split, the local part $\pi_v$ is supercuspidal or the principal series $\pi(\chi_1,\chi_2)$ where $\chi_1\cdot\chi^\sigma=1.$
%can not be the principal series $\pi(\mu_1,\mu_2),$
%where $\mu_1\neq\mu_2$ and $\mu_1|_{F^\times}=\mu_2|_{F^\times}=1.$
%%if $\Sigma_v $ is an infinitely-dimensional representation, 
%%$\pi^D_v =\mu_v^c/\mu_v$  if 
%%$\Sigma_v$ is an one-dimensional representation.
%\end{enumerate} 
%\end{thm}
\begin{proof}[\textbf{Proof of the Main Theorem (Global)}]
%\paragraph*{Proof of the Main Theorem(Global)}%\autoref{globalmain}
%Assume $\pi^D$ is $D^\times(\mathbb{A})$-distinguished, then 
Since $\omega_{\pi^D}|_{Z(\mathbb{A})}=1,$ by Hilbert's Theorem 90, the central character $\omega_{\pi^D}=\mu^\sigma/\mu$
for some Hecke character $\mu$ of $\mathbb{A}_E^\times.$
%And $\Sigma^D=\pi^D\otimes\mu\boxtimes\mu|_{\mathbb{A}^\times}$ is $PD^\times$-distinguished.

It is clear that $\Sigma^D$ is $PD^\times(\mathbb{A})$-distinguished if and only if $\pi^D$ is $D^\times(\mathbb{A})$-distinguished.
\par
Now we prove $(ii)$ implies $(iii).$
Assuming that $\Sigma^D$ is $PD^\times(\mathbb{A})$-distinguished, then $\theta_{X_D,W,\psi}(\Sigma^D)$ has a nonzero $(N,\psi)$-Whittaker model due to \cite[Proposition 5.2]{ganshimura}, i.e.
$\tau=\theta_{X_D,W,\psi}(\Sigma^D)\neq0,$ which implies $L^S(1,\Sigma^D,Std)=\infty$ due to \cite[Theorem 2]{yamana2014Lfun}. Then the partial $L$-function $L^S(s,\Sigma,Std)$ has a pole at $s=1$ as well.
By Proposition \ref{nonvanishing}, $\Sigma$ is $\PGL_2(\mathbb{A})$-distinguished.
Moreover $\tau_v$ is both $\psi_v$-generic and  $\psi_{v,\epsilon_v}$-generic. 
%And the partial $L$-function $L^S(s,\Sigma^D)$ and the complete $L$-function $L(s,\Sigma^D)$ has the same analytic property at $s=1.$
%The identity above  is equivalent to the condition $L(1,\Sigma^D)=\infty.$ 
%Since $L^S(s,\Sigma^D)$ and $L^S(s,\Sigma)$ has the same analytic property, i.e. $L^S(s,\Sigma)$ has a simple pole at $s=1.$
%Moreover, the global theta lifting $\theta_\psi(\Sigma^D)$ has a Euler product expression $\theta_\psi(\Sigma^D)\cong \otimes_v \theta_{\psi_v}(\Sigma^D_v)$ and at each local place, the local theta lifting $\theta_{\psi_v}(\Sigma^D_v)$
%is $(N_v,\psi_v)$-generic. 
By the local results, for each place $v$ of $F$
where $D_v$ is ramified and $E_v$ is not split, we have $\Sigma^D_v=\Sigma_v$ and
$0\neq \Hom_{N_v}(\theta_{X_{D,v},W_v,\psi_v}(\Sigma^D_v),\psi_v)=\Hom_{N_v}(\theta_{V_v,W_v,\psi_v }(\Sigma_v),\psi_{v,\epsilon_v} ),$ where $\epsilon_v\in F_v\setminus N_{E_v/F_v}E_v^\times. $
This means $\Sigma_v$ can be written as $\BC(\tau_0)\otimes\omega_{\tau_0}\omega_{E_v/F_v}$ for some representation $\tau_0$ of $\GL_2(F_v)$ and $\tau_0|_{\GL_2^+}$ is irreducible.
%Hence, the representation $\pi_v=\pi^D_v$ must be supercuspidal or Steinberg representation or
%the principal representation $\pi(\chi_1,\chi_1)$ with $\chi_1\chi_2^\sigma=1.$
%if $\Sigma^D_v$ is of infinitely dimensional.
\par
Finally, we prove that $(iii)$ implies $(ii).$ Assume that the local conditions hold.
By the local results, we know $\theta_{X_{D_v},W_v,\psi_v}(\Sigma^D_v)\neq0$ and  is both $\psi_{v,\epsilon_v}$-generic and $\psi_v$-generic for each $v$ where $D_v$ ramified,  $E_v$ non-split and $\epsilon_v\in F^\times_v \setminus N_{E_v/F_v}E_v^\times.$ % Then $\theta_{V_v,W_v,\psi_v}(\Sigma_v)\neq0.$
 If $\Sigma$ is $\PGL_2(\mathbb{A} )$-distinguished, then  $\tau=\theta_{V,W,\psi}(\Sigma)$ is $\psi$-generic as a cuspidal representation of $\GSp^+(W)$ and the partial L-function $L^S(s,\Sigma,Std)$ has a pole at $s=1$. By computing the constant term and Fourier coefficient of $\theta_{W,X_D,\psi }(\tau) ,$
  we can obtain that
$\Sigma^D=\theta_{W,X_D,\psi}(\tau)$ is a nonzero cuspidal automorphic representation of $\GSO(X_D,\mathbb{A} ),$ 
and then $\Sigma^D$ is $PD^\times(\mathbb{A})$-distinguished by \cite[Proposition 5.2]{ganshimura} and the fact $\tau$ is $\psi$-generic.
\end{proof}